\documentclass[reqno,12pt]{amsart}
\usepackage[utf8]{inputenc}  %upside down !

\usepackage{float} % H means exactly here
\usepackage[in]{fullpage}

\newcommand{\ds}{\displaystyle}

\newcommand{\tensor}{\otimes}

\newcommand{\op}{\mathcal}

\newcommand{\cdc}{,\dots,}

\DeclareFontFamily{U}{mathx}{}
\DeclareFontShape{U}{mathx}{m}{n}{<-> mathx10}{}
\DeclareSymbolFont{mathx}{U}{mathx}{m}{n}
\DeclareMathAccent{\widehat}{0}{mathx}{"70}
\DeclareMathAccent{\widecheck}{0}{mathx}{"71}

\usepackage{amsmath}%
\usepackage{amsthm}
\usepackage{amsfonts}%
\usepackage{amssymb}%
\usepackage{graphicx}
\usepackage{xy,amsthm,enumerate,xypic,array}  %xypic,eufrak}
\usepackage{xcolor}

\input{xy}
\xyoption{all}

\numberwithin{equation}{section}

\newtheorem{theorem}{Theorem}[section]
\theoremstyle{plain}

\newtheorem{corollary}[theorem]{Corollary}
\newtheorem{lemma}[theorem]{Lemma}
\newtheorem{proposition}[theorem]{Proposition}

\theoremstyle{definition}
\newtheorem{definition}[theorem]{Definition}

\newtheorem{remark}[theorem]{Remark}

\allowdisplaybreaks[2]

\addtocounter{MaxMatrixCols}{2}

\begin{document}
\title{Representation Stability for Marked Graph Complexes} 
%\author{Ben Ward}
\author{Enoch Fedah and Benjamin C.\ Ward}
\email{benward@bgsu.edu}

%\date{December 20, 2010 }
%\subjclass{Primary , ?; Secondary ?, ?} 
%\keywords{Keyword one, keyword two etc.}
\maketitle

\begin{abstract}
We prove a sharp representation stability result for graph complexes with a distinguished vertex, and prove that the chains realizing this sharp bound pass to non-trivial families of graph homology classes.  This result may be interpreted as a higher genus generalization of Hersh and Reiner's stability bound for configuration spaces of points in odd dimensional Euclidean space.  % Whitehouse modules and \Delta_{1,n}
\end{abstract}
\tableofcontents

\section{Introduction.}
This paper concerns a type of graph complex computing pieces of the compactly supported cohomology of the moduli spaces of marked curves.  The graphs in question come with a choice of distinguished vertex  along with a given number of adjacent marked flags.  The chain complex formed by graphs of genus $g$ with $n$ legs and $r$ marked flags will be denoted $B(g,n,r)$.

The graph complexes $B(g,n,r)$ were developed by Payne and Willwacher \cite{PW} who describe the weight 11 compactly supported cohomology of the moduli space  $\op{M}_{g,n}$ in terms of the cohomology of $B(g,n,11)$.  Subsequently, Canning, Larson, Payne and Willwacher establish a parallel relationship between the weight 15 component of $H_c^\ast(\op{M}_{g,n})$ and the cohomology of $B(g,n,15)$.  The computation of the cohomology of these chain complexes is a difficult problem in general and while the difficulty increases with $g$ and $n$, it decreases with $r$.  Organizing the computations by their complexity or ``excess'', both \cite{PW} and \cite{CLPW} compute the homology in the first few examples.  Comparing these cases, \cite{CLPW} observes that ``the low-excess computations of \cite{PW} in weight 11 carry over almost unaltered'' from the $r=11$ to the $r=15$ case.  

A question motivating this article is, then, how low is low?  In other words, % as $n$ and $r$ increase in tandem, at what point 
for which $(g,n,r)$ is the cohomology of $B(g,n,r)$ a formal consequence of the cohomology of $B(g,n-1,r-1)$?  In this article we address this question using the notion of representation stability. %, and prove that our answer is sharp.

Representation stability was introduced by Church and Farb in \cite{CF}.  Recall that to say a sequence   $ A_1 \to A_{2}\to\dots$ is representation stable means, in particular, that each $A_n$ is a representation of the symmetric group $S_n$ and that there exists an $N$ for which $n\geq N$ implies that the irreducible decomposition of $A_n$ is a formal consequence of that of $A_N$, found by adding $n-N$ boxes to the first row of the Young diagrams indexing the irreducibles in $A_N$.

In our case, for a fixed $g$ and $n-r$, we consider the sequence of chain complexes $\dots \to B(g,n,r) \to B(g,n+1,r+1) \to \dots$ formed by adjoining a marked leg labeled by $n+1$ to a graph in the source.  This operation extends to an $S_n$-equivariant chain map, call it $\psi$, and our first results is:

\begin{theorem}\label{mainthmm}  Fix $g$ and $\ell$, for which $m:= 3(g-1) + 2\ell \geq 0$.  The consistent sequence
	$$\dots \to B(g,n,n-\ell)\tensor V_{1^n} \stackrel{\psi}\to B(g,n+1,n+1-\ell)\tensor V_{1^{n+1}} \to \dots $$
	is representation stable and stabilizes sharply at
	%	consistent sequence of symmetric group representations (op)-representation stable sequence $B(g^\prime,k)^{k-(n-r)}$ stabilizes sharply at
	%	$$
	%	k =	\left\lceil \frac{9\beta_1}{2}\right\rceil +3(n-r) .
	%	$$	
	%which translates to
	$$
	n =	\left\lceil \frac{9(g-1)}{2}\right\rceil +3\ell = \left\lceil \ds\frac{3m}{2} \right\rceil.
	$$
\end{theorem}
The need to tensor with the alternating representation $V_{1^n}$ simply reflects the convention that the marked legs are alternating.  

The fact that this sequence is representation stable should not be too surprising.  This is essentially just a consequence of the fact that as $n$ and $r$ increase in tandem, eventually every graph of type $(g,n,r)$ will be the image under $\psi$ of a graph of type $(g,n-1,r-1)$.  Determination of the sharp bound, on the other hand, is more subtle.  If we fix $g$ and $\ell$, as above, the point at which the set of graphs stabilize is $m=3(g-1) + 2\ell$.  This formally implies that the sequence stabilizes at or before $2m$.  The actual bound, however, is $\lceil 3m/2 \rceil$. 

Theorem $\ref{mainthmm}$ is a statement about chain complexes, but knowing the sharp stability point allows us to deduce information at the level of (co)homology.  (We opt for homological conventions to match the variance of the stability map.)  In what follows, for a partition $\lambda$ of $N$ and an integer $n>N$,  the notation $(\lambda,1^{n-N})$ indicates the partition of $n$ formed by adding $n-N$ blocks of size $1$ to $\lambda$.

\begin{corollary}\label{3h2}  Fix $g,\ell$ and $m$ as above. Let $\lambda = (\lambda_1\cdc \lambda_k)$ be a partition of $N$ with $\lambda_k \geq 2$.  Then for each degree $i$:  %and write $h_{\lambda,n}^i$ for the multiplicity of $(\lambda,1^{n-N})$ in $H_i(B(g,n,n-\ell))$.
	
\begin{enumerate}
	\item If $N>\lceil 3m/2 \rceil$, the multiplicity of $(\lambda,1^{n-N})$ in $H_i(B(g,n,n-\ell))$ is $0$.
	\item If $N= \lceil 3m/2 \rceil$, the multiplicity of $(\lambda,1^{n-N})$ in $H_i(B(g,n,n-\ell))$ is independent of $n$. % equals the multiplicity of $\lambda$ in $H_i(B(g,N,N-\ell))$.% for all $n\geq N$.
\end{enumerate}
\end{corollary}
 These results can be rephrased in terms of usual multiplicity stability as defined in \cite{CF}, by tensoring with the alternating representation.  But we find it more natural here to work in this conjugate setting, stabilizing by adding boxes to the first column instead of the more standard first row.  Note, the statements of the corollary are formal, i.e.\ they would work for any representation stable sequence of chain complexes with the given sharp stability bound, although the second statement is not completely immediate, see Theorem $\ref{3h}$.

We then turn to the question of computing these stable multiplicities in our example.  Here we can additionally leverage the fact that in the sequence of chain complexes $B(g,n,n-\ell)$, the top dimension (dimension $m$ in our conventions) is the last to stabilize.  This both guarantees that additional multiplicities must be stable, some of which must moreover be $0$, as well as gives us an formula computing certain non-zero multiplicities.

\begin{theorem}\label{odd2}
	Fix $g,\ell$ and $m$ as above and let $\lambda = (\lambda_1\cdc \lambda_k)$ be a partition of $N=\lceil 3m/2 \rceil$.  %Then in the top degree $m$:
	\begin{enumerate}
		\item  If $k < \lceil m/2 \rceil$, the multiplicity of $(\lambda, 1^{n-N})$ in $H_i(B(g,n,n-\ell))$ is $0$ for all $i$.
		\item If $k = \lceil m/2 \rceil$ and $i<m$, the multiplicity of $(\lambda, 1^{n-N})$ in $H_i(B(g,n,n-\ell))$ is $0$.
		
		\item If $k = \lceil m/2 \rceil$,  the multiplicity of $(\lambda, 1^{n-N})$ in $H_m(B(g,n,n-\ell))$ is independent of $n$.
	\end{enumerate} 
	Moreover, in this latter case this multiplicity, call it $c_\lambda$, 	
	can be expressed as the following sum of Littlewood-Richardson coefficients:
	$$c_\lambda = \ds\sum_{p=0}^{g-1} ( \ds\sum_{\tau} N_{\lambda^\prime, 2\tau, p}),$$
	where $\tau$ is taken over all partitions of $(m-p)/2$ (if any), where $2\tau$ is the partition of $m-p$ formed by doubling the entries of $\tau$,  and where $\lambda^\prime$ is the partition of $m$ associated to the Young diagram found by erasing the first column of $\lambda$.
\end{theorem}
Here the notation $N_{\lambda^\prime, 2\tau, p}$ denotes the multiplicity of $\lambda^\prime$ in the induced representation $V_{2\tau}\circ V_p$.  We remark that the inner sum is $0$ unless $m$ and $p$ have the same parity and $m\geq p$, although we do allow the empty partition of $0$.  See Section $\ref{stabsec}$ for more detail.

Here is an example. Let $(g,n,r) = (7,8,15)$, we first confirm that this triple lies in the stable range of Theorem $\ref{mainthmm}$, since $n=8 \geq \lceil 3m/2\rceil=6$.  In addition to identifying many multiplicities which must be $0$, Theorem $\ref{odd2}$ tells us the following non-zero multiplicities:
$$3V_{5,1,1^2}\oplus V_{4,2,1^2}\oplus 2V_{3,3,1^2}\hookrightarrow H_{4}(B(7,8,15)).$$
For comparison Burkhardt \cite{Burk} computes $H_{4}(B(7,4,11))\cong 2V_{2,2}$ (in our degree conventions).  So, as expected, the homology of $B(7,8,15)$ can not be viewed as a formal consequence of the homology of $B(7,4,11)$.  In particular we can't deduce $\mathsf{gr}_{15} H_c^\ast(\op{M}_{7,8})$ from $\mathsf{gr}_{11} H_c^\ast(\op{M}_{7,4})$, or vice versa, since the latter corresponds to a triple $(g,n,r)$ falling outside of the stable range guaranteed by Theorem $\ref{mainthmm}$.

Note that while this excess $m=4$ example is close to, if not within, the range of the computations which can be done completely explicitly, our results hold for all $m$.  See Section $\ref{stabsec}$ for more examples.

\subsection{Representation stability of Whitehouse modules.}

If the graph complexes $B(g,n,r)$ were only of interest in cases $r=11$ and $r=15$, then our Theorem $\ref{mainthmm}$ might be of limited interest, since it applies only to finite $(g,n)$ for a fixed $r$.  However, the graph complexes $B(g,n,r)$ are of interest for all $r$.  To give a first indication of this, we consider the case when $g=1$.

 Write $C(\mathbb{R}^3,n)$ for the configuration space of $n$ distinct labeled points in $\mathbb{R}^3$.  Note that this space has (rational) cohomology only in even degrees.  The continuous map which forgets the last point yields a representation stable sequence $\dots \to H^{2i}(C(\mathbb{R}^3,n-1))\to H^{2i}(C(\mathbb{R}^3,n)) \to \dots$ in each degree, and Hersh and Reiner \cite{HR} proved that this sequence stabilizes sharply at $3i$.

The family of $S_n$-modules $H^{2i}(C(\mathbb{R}^3,n))$ is the restriction of a family of $S_{n+1}$-modules, introduced by Whitehouse in \cite{Whitehouse}, see Early and Reiner \cite{ER}.  We use the notation $\mathsf{W}_{n,k}$ for the $S_n$-module whose restriction to an $S_{n-1}$-module coincides with $H^{2(n-k)}(C(n-1,\mathbb{R}^3))\tensor V_{1^{n-1}}$.  %These Whitehouse modules have a topological interpretation via the inclusion $\mathbb{R}^3\hookrightarrow S^3 = SU_2$, see \cite{ER}.  
In this paper we compute:

\begin{theorem} \label{whthm} Let $r\geq 2$.  There is an isomorphism of $S_n$-modules
	$$H_i(B(1,n,r)) \cong \begin{cases} \mathsf{W}_{n,r-1} & \text{ if } i= 2(n-r) \\ 0 & \text{else} \end{cases}$$
\end{theorem}
Theorem $\ref{mainthmm}$ may then be  applied to deduce that the Whitehouse modules $\mathsf{W}_{n,n-\ell}\tensor sgn_n$ form a representation stable sequence, stabilizing at $n=3\ell$.  Furthermore, Theorem $\ref{odd2}$ shows this stability is witnessed by an injection $V_{3^\ell}\hookrightarrow \mathsf{W}_{n,n-\ell}$. It follows that the sequence of restrictions $\downarrow \mathsf{W}_{n,n-\ell}$ stabilizes when $n=3\ell+1$, hence $n-1 =3\ell$, recovering the result of Hersh and Reiner \cite[Theorem 1.1]{HR}.

 This computation has a precursor in \cite{WardStir}, which calculated the homology of a chain complex which arose from considering the Feynman transform of Lie graph homology.  We called said complexes ``Stirling complexes'' because their betti numbers were given by Stirling numbers of the first kind.  These Stirling complexes turn out to be quasi-isomorphic to $B(1,n,r)$, but working with the purely combinatorial resolution  $B(g,n,r)$ is advantageous, particularly when computing the character.%  so the main advance of Theorem $\ref{whthm}$ is the identification of their character.  
 
 \subsection{Outlook} We became interested in studying the Feynman transform of Lie graph homology due to its role in a spectral sequence converging to commutative graph homology constructed in \cite{WardMP}.  When $g=1$ this spectral sequence degenerates at the $E_1$ page, and combining Theorem $\ref{whthm}$ with the results of \cite{WardMP}, \cite{WardStir} allows us to describe commutative graph homology in terms of Whitehouse modules.  Borrowing notation for $\Delta_{g,n}$ from \cite{CGP} we find: 

\begin{corollary}\label{delta2}  The unique non-zero homology group of $\Delta_{1,n}$ is isomorphic to
	$\ds\bigoplus_{k \text{ even} }\mathsf{W_{n,k}}.$
\end{corollary}
%One can also prove an analogous statement for commutative graph homology with twisted coefficients, by summing over odd $k$ instead of even.
From this perspective our interest in the graph complexes $B(g,n,r)$ goes beyond the above mentioned cases, i.e.\ beyond $g=1$, $r=11$ or $r=15$, because the relationship to commutative graph homology extends beyond the genus $g=1$ case.  In particular, for arbitrary genus there is a spectral sequence computing commutative graph homology for which the homology of $B(g,n,r)$ appears as a portion of the $E_1$ page, but with additional contributions corresponding to multiple distinguished vertices, each with potentially higher genus.  Since the differentials in this spectral sequence are equivariant, our results, particularly knowledge of the stable multiplicities, will inform future analysis of this spectral sequence.

 Of particular interest would be to make a connection with the topological constructions of Gadish and Hainaut \cite{GH} who describe a large subspace of the homology of $\Delta_{2,n}$ in terms of the Whitehouse modules.  On the one hand, an interpretation of their work in genus $1$ could be used to give a topological interpretation of Corollary $\ref{delta2}$.  On the other hand, our derivation of this result purely in terms of graph complexes may point toward a possible graph complex interpretation of their genus $2$ results.

\tableofcontents

%%%%%%%%%
\subsection{Acknowledgment} BW would like to thank 
V.\ Dotsenko, N. Gadish, L. Hainaut, S.\ Payne, D.\ Petersen and T. Willwacher, for correspondence and conversations related to these results, and would also like to thank the Simons Foundation for their support via Simons Collaboration Grant no.\ 704658.

\section{Representation stability on chains and passage to homology.}\label{repstabsec}

In this section we establish the needed prerequisites regarding representation stability.  We first adopt the following conventions and notation.  We work over the field $\mathbb{Q}$ of rational numbers throughout.  We write $S_n$ for the symmetric group of permutations of $\{1\cdc n\}$.  We view $S_n\subset S_{n+1}$ in the usual fashion (the subgroup of permutations fixing $n+1$).%., and write $\text{Res}$ and $\text{Ind}$ for the associated restriction and induction functors without further ado.

By a partition of $n$ we mean a finite sequence of non-increasing, positive integers which sum to $n$.  By convention we allow the empty partition of $0$.  Given a partition $\lambda= (\lambda_1\cdc \lambda_k)$ of $n$, we define $|\lambda|:=n= \sum \lambda_i$.   % Farb would denote $\lambda=0$.  
We often refer to an entry $\lambda_i$ as a ``block'' of the partition $\lambda$.  We use the notation $1^n$ to denote the partition of $n$ consisting of $n$ blocks of size $1$.  We use the notation $(\lambda, 1^n)$ to denote the partition of $|\lambda|+n$ formed by adding $n$ blocks of size $1$ to $\lambda$.

Recall that for $n\geq 1$, the irreducible representations of the symmetric group $S_n$ are indexed by partitions of $n$.  We write $V_\lambda$ for the irreducible representation associated to $\lambda$, as is standard \cite{FH}.  Given a partition $\lambda= (\lambda_1\cdc \lambda_k)$ and an integer $n \geq |\lambda|+\lambda_1$ we define the $S_n$ representation $V(\lambda)_n$ to be
$$
V(\lambda)_n := V_{(n-|\lambda|,\lambda_1,\lambda_2,\cdc \lambda_k)}
$$
We write $V(\lambda)$ for the family of all such representations.

\subsection{Recollection of representation stability}

To begin we recall the definition of representation stability, after Church and Farb \cite{CF}.  In this article we only consider the uniform variant of representation stability, see \cite[Definition 2.6]{CF}.    First recall that a consistent sequence of symmetric group representations is the data of a representation $A(n)$ of the group $S_n$ for each $n\geq 1$, along with $S_n$-equivariant maps $\phi_n\colon A(n)\to \text{Res} (A(n+1))$ for each $n$.

\begin{definition}\label{rsdef}  Let $A=\{A(n),\phi_n\}$ %$\{...\to A(n)\stackrel{\phi_n}\to A(n+1)\to ...\}$ 
	be a consistent sequence of symmetric group representations and let $N$ be a natural number.  We say $A$ is (uniformly) %\footnote{\cite{CF} defines uniform and non-uniform variants for $N$ (in)dependent of $\lambda$.  We work with the uniform variant throughout.} 
	representation stable at $N$ if $n\geq N$ implies:
	\begin{enumerate}
		\item[(1)] $\phi_n$ is injective.
		\item[(2)] $A(n+1) = \mathbb{Q}[S_{n+1}]\cdot \phi_n(A(n))$
		\item[(3)] If $A(n) \cong \bigoplus c_{\lambda, n} V(\lambda)_n$ then $A(n+1) \cong \bigoplus c_{\lambda, n} V(\lambda)_{n+1}$, where the integer $c_{\lambda,n}$ denotes the multiplicity of the irreducible representation $V(\lambda)_n$ appearing in $A(n)$. 
	\end{enumerate}
 Furthermore, we say $A$ {\it stabilizes sharply} at $N$ if $A$ is representation stable at $N$ and is not representation stable at $N-1$.	
\end{definition}

%We say $A$ is {\it representation stable} if there exists an $N$ for which $A$ is ``representation stable at $N$''. 

To be precise, we view both summands in axiom (3) as being indexed over those partitions $\lambda$ where $c_{\lambda,n}$ is formally defined.  Thus, given a partition $\lambda=(a_1\cdc a_k)$, the content of axiom (3) says that when $n\geq N$ the integer $c_{\lambda,n}$ may be determined as follows.  First one asks: is $a_1+|\lambda|\leq N$?  If the answer is no then $c_{\lambda,n}=0$.  If the answer is yes then $c_{\lambda,n}=c_{\lambda,N}$.   Given a representation stable sequence which stabilizes sharply at $N$, we define the integer $c_\lambda:=c_{\lambda,N}$ and refer to $c_\lambda$ as the stable multiplicity of $V(\lambda)$.

\begin{remark}\label{nomaps}  Note that axiom (3) in the definition of representation stability doesn't invoke the maps.  Given a symmetric sequence $A=\{A(1),A(2),...\}$ we say the coefficients of $A$ stabilize at $N$ if axiom (3) is satisfied for $n\geq N$.  In such a situation we may also make reference to the stable coefficients $c_\lambda$ as above.  We call this phenomenon ``multiplicity stability'' after \cite{CF}.  Specifically, this is the uniform variant of multiplicity stability, see \cite[Definition 2.7]{CF}.
\end{remark}

\subsubsection{Induced sequence.}\label{indss}
The following class of examples of representation stable sequences will be important for us in what follows.  Let $X$ be an $S_m$-module.  We write $X\circ V_n = Ind_{S_m\times S_n}^{S_{n+m}}(X\boxtimes V_n)$.  Define a consistent sequence $\vec{X}$ by defining 
$$\vec{X}(n)=
\begin{cases}
	X & \text{ if } n=m \\ 
X\circ V_{n-m} & \text{ if } n>m \\ 
0 & \text{else}
\end{cases}
$$
along with maps induced by the $S_n$-equivariant isomorphism $V_{n}\to Res(V_{n+1})$.

Define $w(X)$, the width of $X$, to be the largest $w \leq m$ for which there exists an injective equivariant map $V_w \to Res^{S_n}_{S_w}(X)$.  The width may be equivalently described as the largest number appearing in a partition associated to an irreducible subspace of $X$, or equivalently as the maximum number of columns appearing in a Young diagram associated to an irreducible subspace of $X$.

\begin{lemma}\label{stablem} Given  an $S_m$ representation $X$, the consistent sequence
	$\vec{X}$ is representation stable and stabilizes sharply at $n=m+w(X).$
\end{lemma}
\begin{proof}
	This follows immediately from the Littlewood-Richardson rule.
%$X\circ V_n = Ind_{S_m\times S_n}^{S_{n+m}}(X\boxtimes V_n)$
\end{proof}

\subsubsection{Direct sum.}

\begin{lemma}\label{directsumlem}  Let $X_i$ be a representation of $S_{m_i}$ and let $q_i := w(X_i)+m_i$. %which stabilizes sharply at $q_i$ for $i=1\cdc s$.  
	Define the consistent sequence $\vec{X} = \oplus_i \vec{X}_i$.  Then $\vec{X}$ is representation stable and stabilizes sharply at $\text{max}_i\{q_i\}$. 
\end{lemma}
\begin{proof}
	The fact that a direct sum of representation stable sequences is representation stable is straight-forward.  The fact that a direct sum of representation stable sequences stabilizes at or before the last of its constituent pieces is also straight-forward.  
	
	The only concern, then, is to establish that it doesn't stabilize before hand.  This follows from the fact that the multiplicity of each $V(\lambda)$ is non-decreasing in each $\vec{X}_i$ -- even before hitting the stabilization point, again invoking the Littlewood-Richardson rule.
\end{proof}

\subsubsection{Transposition} \label{altsec}
	The	alternating (aka sign) representations form a consistent sequence via the isomorphism $ V_{1^n} \stackrel{\iota_n}\to Res_{S_n}^{S_{n+1}}(V_{1^{n+1}})$.  This sequence is not representation stable.  However, the induced sequences which appear most naturally for us will be formed by tensoring with this consistent sequence.  We thus introduce the following notation. If $A=\{A(n),\phi_n\}$ is a consistent sequence we define $A^\top$ to be the consistent sequence 
	$$
	\dots\to	A(n)\tensor V_{1^n}\stackrel{\phi_n\tensor \iota_n}\longrightarrow A(n+1)\tensor V_{1^{n+1}}\to\dots.
	$$
	We say that $A$ and $A^\top$ are {\bf conjugate} sequences.  If will be convenient to introduce the following shorthand terminology:
	
	\begin{definition}\label{conjstab}  We say a consistent sequence $A$ is conjugate stable, stabilizing sharply at $n$, if the sequence $A^\top$ is representation stable, stabilizing sharply at $n$.
\end{definition}

\begin{definition}\label{widecheck}
	Given an $S_m$-module $X$, %we define $X^\vee:=X\tensor V_{1^m}$, 
	we define the consistent sequence $\widehat{X}:=(\overrightarrow{X\tensor V_{1^m}})^\top$. 
\end{definition}
In particular %$\widehat{X}(n) = (X\tensor V_{1^m}) \circ V_m)\tensor V_{1^{n+m}}$
$$\widehat{X}(n)=
\begin{cases}
	X & \text{ if } n=m \\ 
	X\circ V_{1^{n-m}} & \text{ if } n>m \\ 
	0 & \text{else}
\end{cases}
$$
Define $\rho(X) = w(X\tensor V_{1^m})$, i.e. $\rho(X)$ is the maximum number of rows appearing in a Young diagram associated to an irreducible summand of $X$.  Lemma $\ref{stablem}$ immediately implies:
\begin{lemma}\label{rhos}  Given an $S_m$-module $X$, the consistent sequence $\widehat{X}$ is conjugate stable, stabilizing sharply at $n=m+\rho(X)$.
\end{lemma}

\subsection{Representation Stability for Chain Complexes.}

We now lift the definition of representation stability from vector spaces to chain complexes.  This is essentially straight-forward.  

First, by a chain complex of $S_n$-modules we mean a chain complex $(A_\bullet,d)$ along with an action of $S_n$ on each vector space $A_i$ for which each $d_i$ is equivariant.  By a consistent sequence of symmetric group representations in the category of chain complexes we mean a chain complex of $S_n$-modules $A(n)$ for every $n\geq 1$, along with a sequence of chain maps $\phi_n\colon A(n)\to A(n+1)$ which is equivariant in each degree.  Specifically this means $\phi_{n,i}\colon A(n)_i \to Res (A(n+1)_i)$ is $S_n$-equivariant for all $i$, where $\phi_{n,i}$ denotes the restriction of $\phi_n$ to degree $i$.

A consistent sequence of chain complexes gives rise to a consistent sequence in each degree and we may ask if such a sequence is stable in the sense of Definition $\ref{rsdef}$.  With this in mind we define:

\begin{definition}\label{rsdefch}  Let $A=\{A(n),\phi_n\}$ be a consistent sequence of symmetric group representations in the category of chain complexes and let $N$ be a natural number.  We say $A$ is representation stable at $N$ if each consistent sequence
$$A_i=\{...\to A(n)_i\stackrel{\phi_{n,i}}\to A(n+1)_i\to ...\}$$	
is representation stable at $N$ in the sense of Definition $\ref{rsdef}$.
\end{definition}

\subsubsection{Passage to homology and multiplicity stability} In general, there is no reason to expect that the homology of a representation stable sequence of chain complexes will itself be representation stable, or even that it would be multiplicity stable. However, as we now show, if we know the sharp stable bound of a representation stable sequence of chain complexes then there are at least {\it some} multiplicities which are guaranteed to stabilize on homology.

\begin{theorem}\label{3h}  Fix $\{A(n),\phi_n\}$, a representation stable sequence of chain complexes which stabilizes sharply at $N$ and let $\lambda = (\lambda_1\cdc \lambda_k)$ be a partition satisfying $|\lambda|+\lambda_1 =N$.  %$N=ka$ be any factorization by positive integers with $k\geq 2$.  
Then for each degree $i$, the multiplicity of $V(\lambda)_n$ in $H_i(A(n))$ is stable (aka constant) for $n\geq N$. 
\end{theorem}
\begin{proof}
	Each $A(n)_i$ contains an equivariant subspace of cycles $Z_i(n)$.  This subspace has an equivariant subspace of boundaries $B_i(n)$.  We can choose (rationally) equivariant complements.  The complement of the cycles maps (via $d$) isomorphically to the boundaries one dimension lower.  The complement of the boundaries is identified with the homology (mapping a cycle to its homology class).  Thus we have an $S_n$-equivariant decomposition:
	\begin{equation}\label{decomp}
		A(n)_i  \cong B_{i-1}(n)\oplus Z_i(n) \cong B_{i-1}(n)\oplus B_i(n) \oplus H_i(n)
	\end{equation}
	The differential with respect to this decomposition has the form $d(a,b,c) = (0,a,0)$.
	
	Let $\lambda$ be a partition of the specified form, i.e.\ a partition such that $(\lambda_1,\lambda)$ is a partition of $N$.   Define $c_{\lambda,n}^i$ to be the coefficient of $V(\lambda)_n$ appearing in $A(n)_i$. Similarly,
	write $z^i_{\lambda,n}$, $b^i_{\lambda,n}$, $h^i_{\lambda,n}$ for the number of copies of $V(\lambda)_n$ appearing in the cycles, boundaries and homology of degree $i$ for the given $n$.   From the Equation $\ref{decomp}$ we know that
	\begin{equation} \label{dim}
		c_{\lambda,n}^i  = b_{\lambda,n}^{i-1} +z_{\lambda,n}^i = b_{\lambda,n}^{i-1} + b_{\lambda,n}^i + h_{\lambda,n}^i.
	\end{equation}

	We then claim that for $n\geq N$, both
\begin{equation}\label{ineq}
	z^i_{\lambda, n} \leq z^i_{\lambda, n+1} \text{ and }b^i_{\lambda, n} \leq b^i_{\lambda, n+1}.
\end{equation}	
	This follows from the injectivity assumption on the map $\phi\colon A(n)_i\to A(n+1)_i$ as follows.  First, restricting $\phi$ to cycles, the injectivity implies that the number of copies of $V(\lambda)_n$ appearing in $Res(Z_i(n+1))$ is greater than or equal to the number of copies of $V(\lambda)_n$ appearing in $Z_i(n)$.  Second, a copy of $V(\lambda)_n$ appearing in $Res(Z_i(n+1))$ could only arise from a partition of $n+1$ found by adding $1$ to an entry of $(n-|\lambda|, \lambda_1\cdc \lambda_k)$, or by adding $1$ as a new, last entry.  However since $A$ stabilizes at $N$, the multiplicity of such a partition in $A(n+1)$ will be zero unless we add $1$ to the first entry.  In other words, $V(\lambda)_n$ can only arise in $Res(Z_i(n+1))$ as the restriction of $V(\lambda)_{n+1}.$  Hence the multiplicity of $V(\lambda)_n$ in $Res(Z_i(n+1))$ is equal to the multiplicity of $V(\lambda)_{n+1}$ in $Z_i(n+1)$, establishing Equation $\ref{ineq}$ for cycles.  The claim for boundaries follows similarly. %  

To conclude,  for $n\geq N$ we know $c_{\lambda, n}^i = c_{\lambda, n+1}^i$, thus by Equation $\ref{dim}$,
	$$
	b_{\lambda,n}^{i-1} +z_{\lambda,n}^i = b_{\lambda,n+1}^{i-1} +z_{\lambda,n+1}^{i}.
	$$
But since both terms on the right are greater than or equal to their respective associate on the left, they must be equal.  In particular $b_{\lambda,n}^{i-1} = b_{\lambda,n+1}^{i-1}$.  But we also know (again from Equation $\ref{dim}$) that 
	$$
	b_{\lambda,n}^{i-1} + b_{\lambda,n}^i + h_{\lambda,n}^i = b_{\lambda,n+1}^{i-1} + b_{\lambda,n+1}^i + h_{\lambda,n+1}^i,
	$$
	from which we conclude that $h_{\lambda,n}^i = h_{\lambda,n+1}^i$.
\end{proof}
Note that the second statement of Corollary $\ref{3h2}$ given in the introduction follows from this result.  To see this, conjugate the partition in the statement of the Corollary and apply Theorem $\ref{3h}$ to the associated stable family.

The above result could be viewed as a (partial) answer to the general question of how representation stability at the chain level constrains the possible multiplicities on homology.  There are additional features in our example of interest which we will exploit to strengthen this statement in that case.

\section{The graph complex $B(g,n,r)$.}  In this section we recall the construction of the graph complex $B(g,n,r)$ from \cite{PW}.  We then prove a few technical lemmas which will be of use in the proof or our main theorem (Theorem $\ref{mainthm}$) in the following section.

\subsection{Conventions for graphs}  To begin we fix terminology and notation regarding graphs. 

\subsubsection{Graphs}
A graph is a 4-tuple $(V,F,a,\iota)$ consisting of a non-empty set $V$, called the vertices of the graph, a set $F$ called the flags of the graph, a function $a\colon F\to V$ called the adjacency map, and an involution $\iota\colon F\to F$.  In this paper we restrict attention to the case when both $V$ and $F$ are finite sets.

The intuition behind this definition views a flag $f\in F$ as half of an edge which is adjacent to the vertex $a(f)$.  Two distinct flags (half-edges), call them $f$ and $f^\prime$, are joined together to form an edge provided that $f^\prime=\iota(f)$.  The fixed points of the involution, i.e.\ those flags having $f=\iota(f)$, are not considered edges.  They are called legs, and can be represented graphically as half-edges, terminating only at one vertex.  Given a graph, we write $E$ for the set of edges and $L$ for the set of legs.  In particular $|L|+2|E| = |F|$.   Note that we allow tadpoles, i.e.\ edges consisting of two flags which are both adjacent to the same vertex.  We also allow parallel edges, i.e.\ multiple edges adjacent to the same pair of vertices (or single vertex in the case of parallel tadpoles).

A morphism of graphs is a pair of functions $V\to V^\prime$ and $F\to F^\prime$ commuting with the adjacencies and involutions.  If both such functions are bijective, the morphism is an isomorphism.

The connected components of a graph are the blocks of the partition of the vertices determined by the equivalence relation generated by declaring $v_1\sim v_2$ if there exists an edge $\{f,\iota(f)\}$ for which $a(f)=v_1$ and $a(\iota(f))=v_2$.  For the remainder of this paper we assume all graphs have one connected component, i.e.\ are connected.  Given such a graph with vertices $V$ and edges $E$ we define $\beta:=|E|-|V|+1$.

\subsubsection{Marked Graphs}\label{markedgraphsec}

All of the graphs in this paper will come with an additional decoration which we call a marking.
\begin{definition}  A marking $(v,D)$ of a graph $(V,F,a,\iota)$ is a choice of a vertex $v\in V$ along with the choice of a subset $D$ of the flags adjacent to $v$, i.e.\ $D\subset a^{-1}(v)$.  Given a marking, the vertex $v$ is called the distinguished vertex (often ``DV'' for short), the flags of $D$ are called the marked flags, and the vertices in $V\setminus v$ (if any) are called the neutral vertices.
	
We remark that given a marking $(v,D)$ for which $D$ is not empty, we can read off the distinguished vertex as the unique element of $a(D)$.  However, $D$ may be empty.	
	
A marking $(v,D)$ is called admissible if it satisfies the following conditions:
\begin{enumerate}
	\item (stability) Every neutral vertex $w\neq v$ satisfies $|a^{-1}(w)| \geq 3$,
	\item (no neutral tadpoles)  If $\{f,\iota(f)\}$ forms a tadpole then $a(f)=a(\iota(f))=v$,
	\item (no double marked tadpoles)  If $\{f,\iota(f)\}$ is a subset of $D$ then $f = \iota(f)$. 
\end{enumerate}
\end{definition}

An unlabeled marked graph is a graph (connected by assumption) along with an admissible marking.  A morphism of unlabeled marked graphs is assumed to preserve the distinguished vertex and set of marked flags.

A marked graph is an unlabeled marked graph along with a total order on $L$, the set of legs of the graph.  This total order will be manifest as a bijection $L \stackrel{\eta}\to \{1\cdc |L|\}$.  A morphism of marked graphs is a morphism of unlabeled marked graphs which preserves the given total order.  In particular, an isomorphism of marked graphs must preserve the leg labeling.

A marked graph will typically be denoted by a single character, $\gamma$ say.  We then refer to its constituent pieces by the above letters $\gamma=(V,F,\iota,a,v,D,\eta)$, with subscript ($V_\gamma$, $F_\gamma$, etc) as needed.  We similarly refer to derived information by a subscript, e.g. $L_\gamma$ for the set of legs of $\gamma$ as well as $\beta_\gamma$, $E_\gamma$, etc..

\begin{definition}\label{markedgraphdef} A marked graph $\gamma$ is of type $(g,n,r)$ provided $g=\beta_\gamma+1$, $|L_\gamma|=n$ and $|D_\gamma|=r$.  We write $\Gamma(g,n,r)$ for the set of such graphs.
\end{definition}

\begin{figure}
	\centering
	\includegraphics[scale = .8]{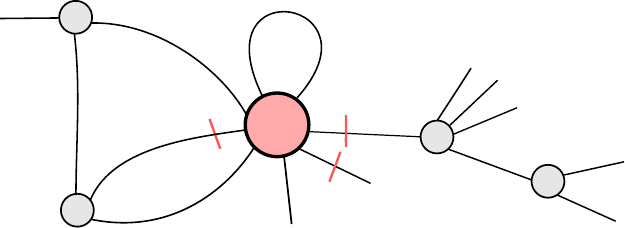}
	\caption{A marked graph of type (4,8,3) with leg labeling suppressed.}
	\label{fig:mg2}
\end{figure}

We call $g$ the genus of the graph.  This shift by $+1$ when comparing $g$ and $\beta$ is, at this point, a convention which may be thought of as asserting that the distinguished vertex $v$ has itself genus $1$.  This convention comes from the examples of interest, as we will see below.

Our convention for drawing marked graphs is to color the distinguished vertex red, and to place red tic-marks on the marked flags.  See Figure $\ref{fig:mg2}$.

\subsubsection{Determinants}

Given a finite set $X$ we define $det(X)$ to be the top exterior power of the $\mathbb{Q}$-span of $X$.  Specifically, this is a one-dimensional vector space concentrated in degree $|X|$, and carrying the alternating action of the symmetric group $S_X$.  We also define $det^{-1}(X):= \Sigma^{-2|X|} det(X)$, i.e.\ the same underlying vector space, but concentrated in degree $-|X|$.  Note the set $X$ could be empty, with the resulting $det(X)$ concentrated in degree 0.

Given a marked graph $\gamma$, with edge set $E=E_\gamma$ and distinguished flags $D=D_\gamma$ we define the vector space
$$
det(\gamma) :=  det(E)\tensor det^{-1}(D).
$$
\subsubsection{Edge Contraction}
Given a marked graph $\gamma$ and an edge $e=\{f,f^\prime\}$ of $\gamma$ we define an associated edge contraction map $\partial_e$ as follows.  First, if $e$ is a tadpole then we define $\partial_e=0$.

Second, if $e$ is an edge which is not a tadpole and if both flags $f$ and $f^\prime$ are unmarked (i.e.\ do not belong to $D$), we first define the marked graph $\gamma/e$ by contracting the edge $e$.  The edge contraction specifies a surjection on the vertices, and the distinguished vertex $\bar{v}$ of $\gamma/e$ is the image of the distinguished vertex $v$ of $\gamma$.  Since both contracted flags were unmarked we may identify the distinguished flags $D_\gamma\subset F_\gamma$ with a subset of the flags of $F_{\gamma/e}$, which form the distinguished flags of $\gamma/e$.  We then define
$$
\partial_e\colon det(\gamma) \to det(\gamma/e)
$$
by asserting that $e$ was dropped in the last position of the wedge product, with all other positions remaining unchanged.

Third, and finally, if $e$ is an edge which is not a tadpole then at most one of the flags $f$ and $f^\prime$ is marked (i.e.\ does belong to $D$).  The case where neither is marked having just been addressed, let us assume without loss of generality that $f$ is marked.   Define the set $N= a^{-1}(\bar{v})\setminus a^{-1}(v)$, i.e.\ $N$ is the set of flags newly adjacent to the distinguished vertex $\bar{v}$ after contracting $e$.  Given an element $f_i\in N$, we define the marked graph $\gamma/e \wedge f_i$ to be the graph formed by contracting $e$, thus discarding $f$, and marking $f_i$ in its place, while retaining the other marked flags.  In particular the resulting set of marked flags may be identified with $(D\setminus{f}) \cup f_i$.

We then define a map
$$
\partial_e\colon det(\gamma) \to \bigoplus_{f_i\in N}det(\gamma/e\wedge f_i)  %),D\setminus{f} \cup f_i)...
$$
by declaring $\partial_e$ removes the edge $e$ in the last position of the wedge product of the edges, while replacing $f$ with $f_i$ in any given position in the wedge product of distinguished flags.

This concludes the definition of $\partial_e$.  For an unmarked flag $f\in a^{-1}(v)\setminus D$ we also define the marking map
$$
\partial_f\colon det(\gamma) \to det(\gamma\wedge f)
$$
which, by definition, adds the flag $f$ in the first (aka leftmost) position of the wedge product.% and then multiplies the result by a factor of $(-1)^{|E_\gamma|}$.

\subsection{The graph complex $B(g,n,r)$}

In this section we define a chain complex of marked graphs, which we denote $B(g,n,r)$.  This construction is due to Payne and Willwacher \cite{PW}, with a few minor differences.  First, the linear dual (edge contraction) differential is more natural for our purposes.  Second, a shift of degree will be needed to make the stabilization maps of degree 0.  Finally,  the emphasis of {\it loc.\ cit.\ }is for the case $r=11$.  So the reader making a direct comparison should observe that what is called $B(g,n)$ in \cite{PW} is the linear dual of what we call $\Sigma^{-n} B(g,n,11)$. %  Further context for the literature related to these graph complexes is given below, but our purpose here is to recall the construction.

First, we write $[\gamma]$ for the isomorphism class of a marked graph $\gamma$.  Given integers $g \geq 1$ and $n,r\geq 0$, we then define the graded vector space $$B(g,n,r) = \bigoplus_{\substack{[\gamma] \text{ of type } \\ (g,n,s) \text{ with } s\geq r}} \Sigma^{n} det(\gamma)_{Aut(\gamma)},$$
where $\Sigma^{n}$ denotes an upward shift in degree by $n$.  We emphasize that all graphs here are marked, which in particular means leg-labeled, and that all (auto)morphisms preserve these extra decorations by assumption.

We then make this graded vector space a chain complex by introducing the differential induced on each summand $[\gamma]$ via the map
$$
\partial = \sum_{e\in E_\gamma} \partial_e + \sum_{f\in a^{-1}(v)\setminus D} (-1)^{|E_\gamma|} \partial_f.
$$

\begin{lemma} $\partial^2=0$.
\end{lemma}
\begin{proof}
The fact that $\partial^2=0$ can be seen, as in \cite{PW}, by restricting to a summand indexed by $\gamma$ and arguing that the terms in $\partial^2$ each appear twice with opposite sign.

For example each pair of unmarked edges of $\gamma$ indexes a pair of terms in $\partial^2$ with opposite sign, corresponding to the two possible orders of the edge contractions.  Similarly each pair of unmarked flags at the distinguished vertex indexes a pair of terms in $\partial^2$ corresponding to the two possible orders in which the flags can be marked.

Slightly more complicated are terms corresponding to an unmarked edge and a flag $f$ adjacent to a neutral vertex which is itself connected to the DV by an edge $e$ .  We can mark $e$, then contract it and mark $f$.  Or we can first contract $e$ then mark $f$.  The sign factor of 
$(-1)^{|E_\gamma|}$
assures that these two terms carry opposite sign reflecting the number of graph edges present upon applying the marking.

Finally, terms corresponding to contracting a marked edge followed by another edge contraction are handled similarly to the case of two unmarked edges.  We observe that if we contract a marked edge $e$ followed by a newly adjacent and now marked edge $e^\prime$, the possible recipients of a marking will coincide with the result of first contracting $e^\prime$ and then $e$.
\end{proof}

%Notation $B(g,n,r)$ means graphs of genus $g-1$ with $n$ legs and at least $r$ marked legs at the DV.  Edges are odd, ala the $\mathfrak{K}$-twisted Feynman transform.

A few remarks about the chain complex $B(g,n,r)$ are in order.  First, note that our degree conventions are homological.  Degrees are encoded combinatorially by declaring that edges have degree $1$ and marked flags have degree $-1$, along with a global shift by $n$ which is a convention chosen to ensure that our stabilization maps are degree $0$.%\footnote{Could combine the shift in degree with the sgn rep to say, like $\mathsf{sB}(g,n,n-\ell))$ is stable. }

Second, observe that the chain complex $B(g,n,r)$ is bounded. The maximum degree occurs when edges are fully expanded and there are the minimum number of markings.  Counting flags, we see this occurs when $|F|=2|E|+n = 3(|V|-1)+r$, hence $|E|=3(|E|-|V|)+(n-r)+3$ or $|E|=3(g-1) + (n-r)$.    Thus the maximum degree (after the global shift by $n$) is $3(g-1)+2(n-r).$  On the other hand, the number of markings $s$ must always be less than or equal to the number of edges and legs, $|E|+n$, so the minimum degree is clearly bounded below by $0$ (after our global shift by $n$).

Following \cite{PW} we define the excess of the chain complex $B(g,n,r)$ to be $$m:=3(g-1) + 2(n-r).$$  Then in our conventions, a chain complex $B(g,n,r)$ of excess $m$ is bounded below by $0$ and bounded above by $m$.  In particular $(g,n,r)$ is of negative excess $m<0$ only if $\Gamma(g,n,r)=\emptyset$, and so we restrict attention to triples $(g,n,r)$ having $m\geq 0$ from here on.  We remark that this condition almost implies the classical stability condition for modular operads $2g+n\geq 3$, but there is one exception, namely $(g,n,r) = (1,0,0)$, which our conventions allow.  Specifically $B(1,0,0)\cong \mathbb{Q}$, concentrated in degree $0$.

%The minimum degree is realized (I guess) by a bouquet of circles with half of the things marked.  

\subsection{The Orbit of a Marked Graph}
In this subsection we consider the sub-representation of $B(g,n,r)$ determined by a marked graph $\gamma\in \Gamma(g,n,r)$.  

Given a permutation $\sigma \in S_n$, define the marked graph $\sigma\gamma$ whose underlying unlabeled marked graph coincides with $\gamma$ but whose leg labels are permuted by $\sigma$.  Specifically the leg labeled by $i$ in $\gamma$ is labeled by $\sigma(i)$ in $\sigma\gamma$.

Next, define the group $I_\gamma\subset S_n$ to be those permutations $\sigma$ for which $\sigma\gamma$ and $\gamma$ are isomorphic.  Crucially, an isomorphism is presumed to preserve the leg labels, else this would be a vacuous condition. To see that $I_\gamma$ is a group, note that if $\phi_\sigma \colon \sigma\gamma \to \gamma$ is a leg-label preserving isomorphism, then the same is true for $\phi_\sigma \colon \sigma\tau\gamma \to \tau \gamma$.  Here, our use of the notation $\phi_\sigma$ to refer to both maps reflects the fact that the unmarked graphs underlying the sources and targets of these morphisms have the same underlying sets of flags and vertices -- only the leg labeling is different.

% (abuse of notation -- could compose with the identity ``non-morphism'').% and we define $\sigma_{\gamma\tau} = \sigma_\tau \circ \sigma \gamma

\begin{lemma}\label{indlem}  Let $\gamma\in\Gamma(g,n,r)$.  The vector space $det(\gamma)_{Aut(\gamma)}$ carries an action of the group $I_\gamma$ for which there is a canonical injection of $S_n$ modules $Ind^{S_n}_{I_\gamma}(det(\gamma)_{Aut(\gamma)})\hookrightarrow B(g,n,r)$.
\end{lemma}
\begin{proof}
The group $I_\gamma$ acts on $det(\gamma)_{Aut(\gamma)}$ in the following fashion.  For $\sigma \in I_\gamma$, the choice of an isomorphism of labeled graphs $\sigma\gamma \stackrel{\phi}\to \gamma$ induces a linear map $det(\sigma\gamma)\stackrel{\phi}\to det( \gamma)$. The identity map on flags, although only a morphism of unlabeled marked graphs, nevertheless induces an isomorphism
$det(\gamma)\stackrel{\iota}\to det(\sigma \gamma)$.

The composition
$$det(\gamma)\stackrel{\iota}\to det(\sigma \gamma)\stackrel{\phi}\to det(\gamma)\twoheadrightarrow det(\gamma)_{Aut(\gamma)}$$
is independent of the choice of isomorphism $\phi$.  Indeed given two such, say $\phi_1,\phi_2$, then 
$\phi_1\circ \iota(x) - \phi_2\circ\iota(x) = \phi_1\circ\phi^{-1}_2(\phi_2 (\iota(x)) - \phi_2(\iota(x))$
is zero in the coinvariants since $\phi_1\circ \phi_2^{-1}\in Aut(\gamma)$.

We further claim that the above composition depends only on the choice of $Aut(\gamma)$ coinvariant class of the input.  To see this choose $\psi \in Aut(\gamma)$.  By abuse of notation we also write $\psi$ for the induced map $det(\gamma) \to det(\gamma)$.  There is a corresponding $\psi_\sigma\in Aut(\sigma\gamma)$, chosen to commute with $\iota$.  Choose also $\phi$ as above, then
$$
\phi \iota (x-\psi(x)) = \phi \iota(x) - \phi \iota\psi(x) = \phi \iota (x) - \phi \psi_\sigma (\iota(x)) = \phi \iota(x) - \phi\psi_\sigma \phi^{-1}\phi \iota(x),$$
with $\phi\psi_\sigma\phi^{-1}\in Aut(\gamma)$, hence 
is zero in $Aut(\gamma)$-coinvariants.  Hence $\sigma \in I_\gamma$ gives a map $det(\gamma)_{Aut(\gamma)} \to det(\gamma)_{Aut(\gamma)}$, and it is straight forward to confirm this is a group action.

Finally, note that there is an $I_\gamma$-equivariant map $det(\gamma)_{Aut(\gamma)} \to B(g,n,r)$ given by inclusion to the summand of the target indexed by $[\gamma]$.  The adjoint of this map is the stated $S_n$-equivariant map.  The fact that this map is injective follows from the fact that no two representatives of distinct cosets in $S_n/I_\gamma$ are isomorphic as labeled graphs.
\end{proof}

\subsection{Edge Cutting.}\label{ecsec}

We conclude this section with a technical lemma (Lemma $\ref{rowlem}$).  The proof of our Main Theorem (Theorem $\ref{mainthm}$) will be by induction on the genus.  This will require a way to compare the stability bound for graphs in genus $g$ with those in genus $g-1$.  For this comparison we study the operation of cutting apart a non-disconnecting edge in a graph.

First, for any marked graph $\gamma$ we define $\rho_\gamma := \rho(Ind_{I_\gamma}^{S_n}(det(\gamma)_{Aut(\gamma)}))$, with $\rho$ as defined in Subsection $\ref{altsec}$.  Then suppose $\gamma$ is a marked graph of type $(g,n,r)$ with $g\geq 2$.   Since $g\geq 2$, we may choose an edge $e$ whose removal doesn't disconnect the graph, and form a new graph $\gamma_c$ by cutting $e$ and labeling the resulting flags by $n+1$ and $n+2$ in either order.  %The purpose of this subsection is to prove the following result

\begin{lemma}\label{rowlem}  For any such $\gamma$ and non-disconnecting edge $e$, we have $\rho_\gamma\leq \rho_{\gamma_c}$. 
\end{lemma} 
\begin{proof}
	We use the following short hand notation: $I:=I_\gamma$ and $I_c:=I_{\gamma_c}$.  We then define $J:=I\cap I_c$.  Observe that $J$ can be equivalently defined in either of the following two ways.  First, $J$ is the subgroup of $I$ consisting of those permutations $\sigma$ for which there exists an isomorphism $\sigma\gamma \to \gamma$ which fixes both flags belonging to the edge $e$.  Second, $J$ is the subgroup of $I_c$ which fixes $n+1$ and $n+2$, i.e.\ $J = I_c\cap S_n$.

	{\bf Step 1:}  There is an injection between the cosets $ S_n/J \to S_{n+2} /I_c, $
	induced by the inclusion $S_n\to S_{n+2}$. % This is the non-normal variant of the second isomorphism theorem.
	This injection on cosets induces an injective $S_n$-equivariant map
	\begin{equation}\label{inj1}
		Ind_{J}^{S_n} Res^{I_c}_{J} det(\gamma_c)_{Aut(\gamma_c)} \hookrightarrow Res^{S_{n+2}}_{S_n} Ind_{I_c}^{S_{n+2}} det(\gamma_c)_{Aut(\gamma_c)} 
	\end{equation}
	where we view $det(\gamma_c)_{Aut(\gamma_c)}$ as an $I_c$ module as in the proof of Lemma $\ref{indlem}$.  We remark that the existence of this injection can be seen by restricting \cite[Proposition 7.3]{Serre} to the double coset represented by the identity element, so we omit a detailed verification.

	{\bf Step 2:} The counit of the adjunction $Ind_{J}^{I} \dashv Res_{J}^{I}$ is a surjective map $$ Ind_{J}^{I}
	( Res_{J}^{I}( det(\gamma)_{Aut(\gamma)})) 	\twoheadrightarrow	det(\gamma)_{Aut(\gamma)}.$$  Applying the functor $Ind_I^{S_n}$ to this surjection results in a surjection:
	\begin{equation}\label{surj1}
		Ind_{J}^{S_n}(Res_{J}^{I} (det(\gamma)_{Aut(\gamma)}))  =	Ind_{I}^{S_n}(Ind_{J}^{I}
		( Res_{J}^{I} (det(\gamma)_{Aut(\gamma)})) )	\twoheadrightarrow	Ind^{S_n}_I(det(\gamma)_{Aut(\gamma)})
	\end{equation}

	{\bf Step 3:}  View the group $J$ as the permutations which have an iso $\sigma\gamma \to \gamma$ which fixes $e$.  From this we establish a surjection
	$Res_{J}^{I_c}(det(\gamma_c)_{Aut(\gamma_c)}) \twoheadrightarrow
	Res_{J}^{I}( det(\gamma)_{Aut(\gamma)}),
	$
	hence a surjection
	\begin{equation}\label{iso1}
	Ind_{J}^{S_n} ( Res_{J}^{I_c}(det(\gamma_c)_{Aut(\gamma_c)}) \twoheadrightarrow
		Ind_{J}^{S_n} ( Res_{J}^{I}( det(\gamma)_{Aut(\gamma)}).
	\end{equation}

	{\bf Step 4:}  The number $\rho_{\gamma_c}$ counts the maximum number of rows appearing in an irreducible subspace of $Ind_{I_c}^{S_{n+2}} det(\gamma_c)_{Aut_{\gamma_c}}$.  This number is greater than or equal to the number of rows appearing in the target of morphism $\ref{inj1}$, which is in turn greater than or equal to the number appearing in the source of morphism $\ref{inj1}$, since this morphism is an injection.  This latter number is also the number of rows appearing in the source of Equation $\ref{iso1}$ and hence greater than the number of rows appearing in the target, since this morphism is a surjection.  This latter number is also the number of rows appearing in the source of Equation $\ref{surj1}$.  Since the morphism in Equation $\ref{surj1}$ is also surjection, this latter number is greater than the number of rows appearing in the target which, by definition, is $\rho_\gamma$.  
\end{proof}

\section{Representation stability for $B(g,n,r)$.}  

The goal of this section is to prove our main theorem -- that the graph complexes $B(g,n,r)$ are representation stable with the given sharp stability bound.  To begin, we define the stabilization map.

Given a marked graph $\gamma \in \Gamma(g,n,r)$ we may form a new marked graph, call it $\gamma\wedge l_{n+1} \in \Gamma(g,n+1,r+1)$ by adding a new marked leg to $\gamma$ at the distinguished vertex which is labeled by $n+1$.  Under the convention that this marked leg is added in the last position, this correspondence induces an isomorphism $det(\gamma)\to det(\gamma\wedge l_{n+1})$.  This isomorphism is equivariant with respect to the automorphism group actions on both sides (note the automorphism groups of these two graphs are canonically isomorphic).  Extending linearly defines a map $\psi_{g,n,r}\colon B(g,n,r) \to B(g,n+1,r+1)$.  Note that by our degree conventions each $\psi_{g,n,r}$ is degree $0$, and in addition preserves the top non-zero degree of the chain complexes, namely $m:=3(g-1)+2(n-r)$.

%by adding a marked leg with label $n+1$ to the DV of the input. 

\begin{lemma}\label{consist}  For each $g$ and $\ell$,
$$\dots\to B(g,n,n-\ell) \stackrel{\psi}\to B(g,n+1,n+1-\ell) \to \dots$$
is a consistent sequence of chain complexes.
\end{lemma}
\begin{proof}
The fact that the maps $\psi_{g,n,r}$ are $S_n$ equivariant is immediate.  It remains to verify that each $\psi_{g,n,r}$ is a chain map.  The differential is given by a sum over edge contractions and flag markings, so since $\psi$ changes neither the set of edges nor the set of unmarked flags, the differential terms correspond.  It thus remains to verify that each differential term commutes with $\psi$.  The only subtlety in this verification lies in verifying that our sign conventions were chosen such that $\partial_f \psi = \psi \partial_f$.  But since $\partial_f$ places $f$ in the first/leftmost position of the wedge product and $\psi$ adds a leg $l_{n+1}$ in the last/rightmost position, both are equal to $f \wedge - \wedge  l_{n+1}$. 
\end{proof}

This is the family of consistent sequences whose representation stability, after tensoring with the sign representation, will be our subject of our main theorem:

\begin{theorem}\label{mainthm}  Fix $g$ and $\ell$ with $m:=3(g-1) +2\ell \geq 0$.  The consistent sequence
	$$\dots \to B(g,n,n-\ell)\tensor V_{1^n} \stackrel{\psi}\to B(g,n+1,n+1-\ell)\tensor V_{1^{n+1}} \to \dots $$
	is representation stable.  Furthermore, this representation stable sequence stabilizes sharply at
%	consistent sequence of symmetric group representations (op)-representation stable sequence $B(g^\prime,k)^{k-(n-r)}$ stabilizes sharply at
%	$$
%	k =	\left\lceil \frac{9\beta_1}{2}\right\rceil +3(n-r) .
%	$$	
%which translates to
	$$
	n =	\left\lceil \frac{9(g-1)}{2}\right\rceil +3\ell = \left\lceil \frac{3m}{2}\right\rceil.
	$$

\end{theorem}

The remainder of this section will be dedicated to the proof of Theorem $\ref{mainthm}$.

\subsection{Core Graphs}
To prove Theorem $\ref{mainthm}$ we will decompose each $B(g,n,n-\ell)_i$ over what we call core graphs.   To begin the proof we introduce and study this notion.  Note that the terminology ``core graph'' appears in the literature in other contexts, but here-in it is defined by:

\begin{definition}  An unlabeled marked graph $\xi$ is called a core graph if it has no marked legs.  
\end{definition}

Let us define a family of core graphs which will play a key part in the identification of the sharp stability point.  Fix $g$ and $\ell$ with $m:= 3(g-1)+2\ell\geq 0$.  Notice that $g$ and $m$ have opposite parity.  For each non-negative integer $p$ having the opposite parity as $g$ (hence the same parity as $m$) and satisfying $p < g$ and $p \leq m$ we construct a core graph of type $(g,m,m-\ell)$ via the following four steps:
\begin{enumerate}
	\item Connect a distinguished vertex $v_0$ to $p$ neutral vertices with a pair of parallel edges to each.  Add a leg at each of these $p$ neutral vertices, ensuring stability.
	\item Define $t = (g-1-p)/2$, a non-negative integer by assumption.  Connect the distinguished vertex to $t$ additional neutral vertices by $3$ edges each.
	\item Define $y = (m-p)/2$, a non-negative integer by assumption. Connect $v_0$ to $y$ additional neutral vertices by a single edge to each.  Then add a pair of legs to each of these $y$ neutral vertices to ensure stability.
	\item Mark each of the $y+2p+3t$ edges of this graph, forming an unlabeled marked graph of type $(p+2t+1, 2y+p, y+2p+3t)$.
\end{enumerate}
The graph thus formed has no marked legs, hence is a core graph, and is easily seen to be of type $(g,m,m-\ell)$.  %We will refer to this construction via the following notation:

\begin{definition}\label{theta}  Fix integers $g,\ell,p$ with $g$ and $p$ of opposite parity and satisfying $0\leq p<g$ and $ p \leq m$.  The marked graph of type $(g,m,m-\ell)$ constructed above is denoted  $\theta_{g,\ell}(p)$.  See Figure $\ref{fig:sat2}$.
\end{definition}

The first manifestation of stability is that eventually (as $n$ increases but $m = 3(g-1)+2(n-r)$ stays fixed) there are no core graphs, as we now establish: 

\begin{lemma}\label{satlemma} If $\xi$ is a core graph of type $(g,n,r)$ then $$ n \leq m := 3(g-1) + 2(n-r).$$  Moreover, $n=m$ if and only if $\xi\cong \theta_{g,\ell}(p)$ for some $p$.% this upper bound is strict in the sense that for every choice of $g$ and $n-r$, there is a core graph of type $(g,n,r)$  satisfying $n=m$. %$$ n = 3(g-1) + 2(n-r).$$% and having no marked legs.
\end{lemma}
\begin{proof}  Let $\xi$ be a core graph of type $(g,n,r)$.
	On the one hand, by our exclusion of doubly-marked tadpoles, if $\xi$ has no marked legs then every marking is on a distinct edge.   Hence $$|F|=2|E|+n = 2|E| + n-r + r \leq 3|E| + (n-r).$$	
	On the other hand, by stability, then using $|E|-|V|=\beta-1=g-2$, gives us
	$$|F| \geq 3(|V|-1) + r = 3|E|-3\beta +r.$$
	Combining these inequalities establishes the first statement.
	
		For the second statement, the ``if'' implication is immediate, so suppose that $m=n$.  This implies that both inequalities above are equalities and hence would require both that $r = |E|$ and that every neutral vertex is trivalent.  Since $\xi$ is a core graph, $r=|E|$ implies that every edge is marked, hence adjacent to the distinguished vertex.  Each neutral vertex must then be connected to the distinguished vertex via 1,2, or 3 edges and adjacent to 2,1, or 0 legs respectively.  The number of such neutral vertices of $\xi$ of each type will correspond to the statistics $y,p,t$ in the construction of $\theta_{g,\ell}(p)$ above.  In particular, let $p$ be the number of neutral vertices connected to the distinguished vertex by a pair of parallel edges.  Then $y = (n-p)/2 =(m-p)/2$ as desired.  
	
	It remains to verify that $t = (g-1-p)/2$.  Suppose that $\xi$ has $q$ tadpoles.  Then the genus of $\xi$ is $2t+q+p+1$, so it suffices to show that $q=0$.  But if $\xi$ had a tadpole, we could remove it to find another core graph of type $(g-1,n,r-1)$ and excess $3((g-1)-1) + 2(n-(r-1)) = n-1$, which contradicts the first statement of the Lemma, as proven above. \end{proof}

\begin{figure}
	\centering
	\includegraphics[scale=.6]{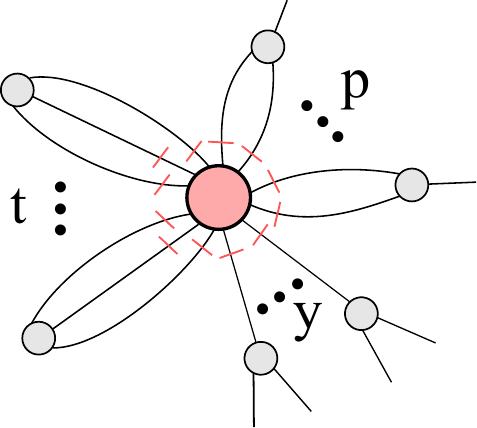}
	\caption{The core graph $\theta_{g,\ell}(p)$ with distinguished vertex in red. The red tic marks denote the marked flags.  The ...\ refers to the number of copies of said feature appearing in the graph.} %Marked graphs having no marked legs and maximum $n$ relative to $g$ and $n-r$.  AKA saturation graphs.}
	\label{fig:sat2}
\end{figure}

\begin{corollary}\label{legscor} Let $\gamma$ be a marked graph of type $(g,n,r)$ with $n > m= 3(g-1) + 2(n-r)$.  Then $\gamma$ has at least $n-m$ marked legs.  %In particular, $\gamma$ supports a summand of $B(g,n,r)$ contained within the $S_{n}$ orbit of the image of $\psi_{g,n-1,r}$.
\end{corollary}

\subsection{Core sequences}  Write $B(g,n,r)_i$ for the degree $i$ piece of the chain complex $B(g,n,r)$.   Core graphs will index a direct sum decomposition of the
consistent sequence
$(B(g,n,n-\ell)_i,\psi)$
as we now indicate.
	
\begin{definition}	
	Given a marked graph $\gamma$ we define the core of $\gamma$, denoted $\epsilon(\gamma)$ to be the core graph formed by forgetting the leg labels and removing any marked legs.  In particular
	$F_{\epsilon(\gamma)} = F_\gamma \setminus (D_\gamma\cap L_\gamma)$, in the notation of Section $\ref{markedgraphsec}$.
\end{definition}
Given a core graph $\xi$ of type $(g,k,s)$ we define the graded $S_k$-module
$$ A_{\xi}:= 
\bigoplus_{\substack{ [\gamma] \in \text{Iso}\Gamma(g,k,s) \\ \text{such that } \epsilon(\gamma) \cong\xi}  } 
\Sigma^{k} det(\gamma)_{Aut(\gamma)}.%\subset B(g,n,r).
$$
Note that $A_{\xi}$ is concentrated in a single degree $i:= |E_\xi|+k-s$.  We then define the consistent sequence $\widehat{A}_\xi$ as above (Definition $\ref{widecheck}$).  In particular for each $n\geq k$,
$\widehat{A}_\xi(n) := A_{\xi} \circ V_{1^{n-k}}$.%, which we also view as concentrated in degree $i$.   %Since the stabilization map $\psi$ preserves the core of a graph, we may view $\widehat{A}_\xi$ as a subsequence of $B(g,k+t,r+t)_i$ for any $r\leq s$.  Moreover we have:

\begin{proposition}\label{directsumprop} %Fix $(g,n,r)$ and degree $i$.  
	The $S_n$-module $B(g,n,r)_i$ decomposes as $B(g,n,r)_i \cong \oplus \widehat{A}_\xi(n)$, with direct sum taken over isomorphism classes of core graphs of degree $i$ and type $(g,k,u)$ for $k-u\leq n-r$ and $k\leq n$.	
\end{proposition}
\begin{proof}
	By definition
\begin{equation}\label{bdef}B(g,n,r)_i = \bigoplus_{\substack{[\gamma]  \text{ of type } (g,n,s) \text{ with }\\   s\geq r \text{ and } i=|E_\gamma|+n-s}} \Sigma^{n} det(\gamma)_{Aut(\gamma)}.\end{equation}
	
For each isomorphism class $[\gamma]$ indexing this sum, we can take the core of any representative $\epsilon(\gamma)$, then pass to the isomorphism class of unlabeled graphs $[\epsilon(\gamma)]$.  If $\gamma$ was of type $(g,n,s)$ and degree $i$, then the core graph is of type $(g,n-t,s-t)$ and degree $i$, where $t$ is the number of marked legs on $\gamma$.
Setting $k=n-t$, we can nest the direct sum in Equation $\ref{bdef}$ as:
\begin{equation}\label{bdef2}B(g,n,r)_i = \bigoplus_{\substack{[\xi]  \text{ core graphs }\\ \text{ of type } (g,k,u) \text{ with }  k\leq n,  \\ u+n-k\geq r   \text{ and } i=|E_\xi|+k-u}} \left( \bigoplus_{\substack{[\gamma] \text{ with }\\  \xi\cong\epsilon(\gamma) \\ \text{ and } |L_\gamma|=n}} \Sigma^{n} det(\gamma)_{Aut(\gamma)}\right).\end{equation}

Fix a core graph $\xi$ indexing the outer summand.   To complete the proof it suffices to identify $\widehat{A}_\xi(n) = A_\xi \circ V_{1^{n-k}}$ with the direct sum inside parentheses in Equation $\ref{bdef2}$.  Label the $k$ legs of $\xi$.  Call the resulting marked graph $\eta\in \Gamma(g,k,u)$.  Lemma $\ref{indlem}$ gives an isomorphism 
$$A_\xi \cong Ind^{S_k}_{I_\eta} (det(\eta)_{Aut(\eta)}),$$	 induced by the $I_\eta$-equivariant map $det(\eta)_{Aut(\eta)} \to Res^{S_k}_{I_\eta}(A_\xi)$.

We then add marked legs labeled by $\{k+1\cdc n\}$ to the distinguished vertex of $\eta$.  Call the resulting marked graph $\gamma\in \Gamma(g,n,s)$.  Here $s= n-k+u$, hence $s\geq r$.  The expression inside parentheses in Equation $\ref{bdef2}$ is isomorphic to $Ind^{S_n}_{I_{\gamma}} (det(\gamma)_{Aut(\gamma)})$ (after Lemma $\ref{indlem}$).  Since permuting the $n-k$ marked legs of $\gamma$ results in an isomorphic marked graph, and since these are the only marked legs of $\gamma$ we know  $I_\gamma= I_\eta \times S_{n-k}$.  Thus $Ind^{S_n}_{I_{\gamma}} (det(\gamma)_{Aut(\gamma)}) = Ind^{S_n}_{I_{\eta}} (det(\eta)_{Aut(\eta)}) \circ V_{1^{n-k}}, $ since $S_{n-k}$ acts by the alternating representation on the marked legs.
\end{proof}

\begin{corollary}\label{dscor}  The consistent sequence formed by $B(g,n,n-\ell)_i$ is a direct sum of the consistent sequences $\widehat{A}_\xi$, with direct sum taken over isomorphism classes of core graphs $\xi$ of degree $i$ and of type $(g,k,u)$ with $k\leq n$ and $k-u \leq \ell$.
\end{corollary}
\begin{proof}  This follows from the Proposition along with the observation that the stabilization maps $\psi$ preserve the core of a marked graph.  %The requirement $k-u \leq \ell$ comes from observing that adding $n-k$ marked legs to the $u$ marked flags extant in a given core graph $\xi$ must be greater than or equal to $n-\ell$.
\end{proof}

Combining this corollary with Lemma $\ref{directsumlem}$ proves the first statement of Theorem $\ref{mainthm}$.  The remainder of the proof concerns the stability bound.

\subsection{A stability bound via cutting.}

For a core graph $\xi$, write $\rho_{\xi}: = \rho(A_\xi)$, i.e.\ the maximum number of rows appearing in an irreducible summand of the $S_n$-module $A_\xi$.  Invoking Lemma $\ref{rhos}$ proves:

\begin{lemma}\label{sumlem}  Let $\xi$ be a core graph of type $(g,n,r)$.  The conjugate of the consistent sequence $\widehat{A}_\xi$ is representation stable, stabilizing sharply at $n+\rho_\xi$.
\end{lemma}

We now arrive at the fundamental step of the proof, bounding this expression by induction on $g$.

\begin{proposition}\label{sharpprop} Let $\xi$ be a core graph of type $(g,n,r)$ and let $\widehat{A}_\xi$ be the associated core sequence.  %As above, $\rho_\xi$ is the max number of rows $A_\xi= \widehat{A}_\xi(n)$.  
	Then 
	$$ n +\rho_\xi \leq \ds\lceil\frac{9(g-1)}{2} \ds\rceil + 3(n-r).  $$
\end{proposition}

\begin{proof} % The proof will be by induction on the genus, using the operation of cutting (Subsection $\ref{ecsec}$) to reduce the genus in the induction step.
{\bf Induction on Genus -- Base Case.}  To begin, let $\xi$ be a core graph of type $(1,n,r)$.  This means in particular that the graph underlying $\xi$ is a tree, having a distinguished vertex.  If $r=0$ the statement is trivial, so we may assume that $r>0$.  This in turn implies that $\xi$ has at least one edge, hence at least two vertices.

Define a partial order on the set of vertices of $\xi$ by declaring that $v>v^\prime$ iff the unique shortest path connecting $v$ to the DV contains $v^\prime$.  A vertex is maximal with respect to this partial order only if it is adjacent to a unique edge.  Therefore, by stability considerations, every maximal vertex is adjacent to at least two legs.

Choose an order on the maximal vertices $v_1\cdc v_t$.  Note that $r\leq t$, since every marking lies on an edge, and the graph has no cycles.  Since each maximal vertex must be adjacent to at least two legs, we may choose a leg labeling of $\xi$ so that the maximal vertex $v_i$ is adjacent to the legs labeled by $2i$ and $2i-1$.  Denote the resulting marked graph by $\gamma$.  By construction, the group $I_\gamma\subset S_n$ contains the transpositions $(2i-1 \ 2i)$ for $1\leq i \leq t$ and so contains the group $(S_2)^r$.

By definition $\rho_\xi$ is the maximum number of rows appearing in an irreducible summand of $A_\xi\cong Ind_{I_\gamma}^{S_n}( det(\gamma)_{Aut(\gamma)})$.  Since $(S_2)^r\subset I_\gamma$, this number is not more than the maximum number of rows appearing in an irreducible summand of $Ind_{(S_2)^r}^{S_n} Res^{I_\gamma}_{(S_2)^r} (det(\gamma)_{Aut(\gamma)})$.  Since the $S_2$ action on the given pair of legs at each maximal vertex is trivial, the Littlewood-Richardson rule tells us that the number of rows of the latter term is at most $n-r$.  Thus $\rho_\xi \leq n-r$.

Finally, since $\xi$ is a core graph, it must satisfy 
$n\leq 3(g-1) + 2(n-r)$ by Lemma $\ref{satlemma}$.  Here $g-1=0$, so we conclude $n+\rho_\xi \leq 2(n-r) + (n-r) = 3(n-r)$ as desired.

{\bf Induction Step.}  We now assert the induction hypothesis that the stated bound holds for all core graphs of genus less than a given $g\geq 2$.  Given a core graph $\xi$ of type $(g,n,r)$ with $\beta:= g-1 \geq 1$, we say an edge of $e$ is ``good'' if $e$ contains no marked flag and its removal would not disconnect the graph. To establish the bound we divide the analysis into three cases.

{\bf Case 1:}  $\xi$ contains a good edge.  Call it $e$.   Define $\xi_c$ to be graph formed by cutting $\xi$ at the edge $e$, as in  Subsection $\ref{ecsec}$. %, but here, since $\xi$ is unlabeled, we do not specify a labeling of the two new resulting legs.
Since $e$ is good it contains no marked flags, hence $\xi_c$ contains no marked legs, and hence is a core graph.  Thus the induction hypothesis applies to it, and we know that %the graph $\gamma_c$ couldn't skip over the stability point, i.e.\  
$$ \rho_{\xi_c} + (n+2) \leq 4.5(g-2) + 3(n+2-r)+.5.$$
%(Here the $.5$ is needed only if $\xi_c$ has odd genus, but can be added regardless.)  
By Lemma $\ref{rowlem}$ we know $\rho_\xi \leq \rho_{\xi_c}$, and thus
$$
\rho_\xi+n \leq 4.5(g-2) +3(n+2-r) + .5 -2 = 4.5(g-1)  +3(n-r) 
$$
as desired.

{\bf Case 2:} $\xi$ contains no good edge and contains a tadpole.  Since the tadpole is not a good edge it must be marked and marked only once by our conventions.  Call the tadpole edge $e$.

Let $\xi\setminus e$ be the graph formed by deleting $e$.  The result is still stable because there is no stability condition at the DV.  So $\xi\setminus e$ is an unlabeled marked graph of type $(g-1,n,r-1)$.  Also, $\xi\setminus e$ has no marked legs, hence is a core graph, so applying the induction hypothesis we have
$\rho_{\xi\setminus e} + n \leq 4.5(g-2) + 3(n-r+1)+.5 < 4.5(g-1)+3(n-r)$.  It thus suffices to argue that  $\rho_{\xi}\leq \rho_{\xi\setminus e}$.  They are in fact equal.

For this, label the legs of $\xi$ in any fashion, again calling the result $\gamma$.  This also gives us a leg labeling of $\xi\setminus e$, which we denote by $\gamma\setminus e$.   %By definition $\rho_\xi$ is the maximum number of rows appearing in an irreducible of $Ind_{I_\gamma}^{S_n}(det(\gamma)_{Aut(\gamma)})$ and $\rho_{\xi \setminus e}$ is the maximum number of rows appearing in an irreducible of $Ind_{I_{\gamma\setminus e}}^{S_n}(det(\gamma \setminus e)_{Aut(\gamma \setminus e)})$.
View $Aut(\gamma\setminus e) \subset Aut(\gamma)$ as the subgroup of automorphisms which fix $e$.  Any automorphism of $\gamma$ not appearing in this subgroup can be written in the form $\gamma = \gamma_1\circ \gamma_2$ where $\gamma_2\in Aut(\gamma\setminus e)$ and $\gamma_2$ permutes the marked tadpoles of $\gamma$, but fixes all other edges.  Since marked tadpoles have degree $1-1=0$, permutation of marked tadpoles acts trivially on $det(\gamma)$, from which we conclude $$det(\gamma)_{Aut(\gamma)}\cong det(\gamma)_{Aut(\gamma\setminus e)} \cong det(\gamma\setminus e)_{Aut(\gamma\setminus e)}$$ as representations of the group $I_{\gamma\setminus e}$.  It remains to observe that $I_{\gamma\setminus e}=I_\gamma$.

By definition $I_{\gamma\setminus e} \subset I_\gamma$ is the subgroup of permutations in $S_n$ for which there exists a (leg preserving) isomorphism $\sigma\gamma \to \gamma$ which fixes the edge $e$.  It remains to prove $I_{\gamma} \subset I_{\gamma\setminus e}$.  Given $\sigma \in I_\gamma$ we may pick an isomorphism $\psi\colon \sigma \gamma \to \gamma$.  This $\psi$ necessarily permutes the marked tadpoles of $\gamma$ and we may write $\psi = \psi_1\circ \psi_2$ where $\psi_2\colon \sigma \gamma \to \gamma$ is an isomorphism which fixes all marked tadpoles and $\psi_1\colon \gamma\to \gamma$ is a leg fixing automorphism which fixes all edges except the marked tadpoles.  In particular $\psi_2$ fixes $e$, hence realizes $\sigma \in I_{\gamma\setminus e}$ as desired.

%Let $\rho^\prime$ be the maximum number of rows appearing in an irreducible of
%$$Ind_{I_{\gamma\setminus e}}^{S_n}Res^{I_\gamma}_{I_{\gamma \setminus e}}(det(\gamma)_{Aut(\gamma)}).$$

%Here, we view $I_{\gamma\setminus e} \subset I_\gamma$ as the subgroup of permutations for which there exists an isomorphism $\sigma\gamma \to \gamma$ which fixes the edge $e$, from which we conclude $\rho_\xi \leq \rho^\prime$.  (One could show $I_{\gamma\setminus e} = I_\gamma$, but it's not needed here.)  We may also view  $Aut(\gamma\setminus e) \subset Aut(\gamma)$ as the automorphisms of $\gamma$ which fix $e$, from which we conclude that $\rho^\prime \leq \rho_{\xi\setminus e}$.  Hence $\rho_{\xi}\leq \rho_{\xi\setminus e}$.

{\bf Case 3:} $\xi$ contains no good edge and no tadpole.  %In this case, we establish the bound directly without the need of the induction hypothesis.
Since $\xi$ has no good edge, it has no circuit containing 3 or more distinct edges. %\footnote{By circuit we mean a cycle which traverses each each and vertex at most once.}  
Since $\xi$ has no tadpole, it has no circuit containing only one edge.  Thus every circuit in $\xi$ has exactly two edges and two vertices.   And since every edge belonging to such a circuit must be marked, lest there be a good edge, one of these two vertices must be the DV.  Denote the distinguished vertex by $v_0$, and let $v_1\cdc v_k$ be the set of neutral vertices which are adjacent to the distinguished vertex.  For $1\leq i \leq k$, define $e(v_i)\geq 1$ to be the number of edges connecting $v_i$ to $v_0$.

For $1\leq i \leq k$, define $L_i$ to be the subset of the legs of $\xi$ having the property that the unique shortest path connecting said leg to $v_0$ contains $v_i$.  Each leg not adjacent to $v_0$ is contained in a unique $L_i$, although some $L_i$ may be empty.  Define $b_0$ to be the number of indices $1\leq i \leq k$ for which $|L_i|=0$, define $b_1$ to be the number of indices $1\leq i \leq k$ for which $|L_i|=1$ and  define $b_2$ to be the number of indices $1\leq i \leq k$ for which $|L_i|\geq 2$.  In particular, $k=b_0+b_1+b_2$.  Notice that if $|L_i|=0$ then $e(v_i) \geq 3$ and if $|L_i|=1$ then $e(v_i) \geq 2$.  We divide the rest of proof into subcases when $b_0$ is or is not equal to $0$.

%Define $L^\prime$ to be the set of legs which are not adjacent to $v_0$.  Partition the set $L^\prime$ by declaring that two legs belong to a common block provided we can connect their adjacent vertices with a path that does not intersect $v_0$. The blocks of this partition are indexed by the vertices $v_1\cdc v_k$.  This indexing associates a block to the unique neutral vertex $v_i$ which can connect to the legs in said block via a path not containing $v_0$.  Denote block $i$ by $L_i$.  Strictly speaking, some of these ``blocks'' may be empty, i.e.\ there could be vertices $v_i$ for which $|L_i|=0$ for some $i$.

{\bf Subcase: $b_0\neq 0$.}  Choose a vertex $v_j$ such that $L_j = \emptyset$.  Since every edge of $\xi$ contained in a circuit is adjacent to $v_0$, the $L_j=\emptyset$ condition implies that the only edges adjacent to $v_j$ are those connecting $v_j$ to $v_0$.  Form a new core graph $\xi\setminus v_j$ by removing the vertex $v_j$ and each of the $e:=e(v_j)$ edges adjacent to it.  The core graph $\xi\setminus v_j$ is of type $(g - e +1, n, r-e)$.  By the induction hypothesis $\rho_{\xi\setminus v_j}+n \leq \lceil 4.5(g-e) \rceil + 3(n-r+e)$ which, using $e\geq 3$, is easily seen to imply 
$\rho_{\xi\setminus v_j}+n\leq \lceil 4.5(g-1) \rceil + 3(n-r)$.  
% = \lceil 4.5(g-1) + 4.5 - 4.5e \rceil + 3(n-r)+3e \leq  \lceil 4.5(g-1) \rceil + 3(n-r)
%$
%To verify this last inequality, separate the case $e=3$, in which its an equality, and the case $e\geq 4$, in which case $5-4.5e+3e<0$.

Thus to complete this subcase it suffices to show $\rho_{\xi\setminus v_j}=\rho_\xi$.  This follows similarly to case 2 above, where the role of a marked tadpole is replaced by the set of $e$ marked parallel edges connecting $v_0$ and $v_j$.    Choose a leg labeling of $\xi$, call it $\gamma$.  This in turn induces a leg labeling of $\xi\setminus v_j$, call it $\gamma\setminus v_j$.    Any isomorphism $\psi\colon \sigma \gamma \to \gamma$ must permute the set of vertices $v_i$ satisfying $|L_{i}| = 0$ and $e(v_i)=e$, and can be factored into an isomorphism $\psi_2$ which fixes each edge running between $v_0$ and such a $v_i$, followed by an automorphism $\psi_1$ which permutes only those vertices and the edges adjacent to them.   This allows us to establish $I_\gamma = I_{\gamma \setminus v_j}$ (as in Case 2 above).  Also, since all edges running from $v_0$ to such a $v_i$ are marked,  $\psi_1$ acts trivially on $det(\gamma)$ from which we may establish $det(\gamma)_{Aut(\gamma)} \cong det(\gamma\setminus v_j)_{Aut(\gamma\setminus v_j)}$.

%the legs of $\eta$ and $\xi$ are canonically identified. Choosing a leg labeling one can show $I_\xi=I_\eta$.  To see this choose $\sigma\in I_\xi$. By definition there is a leg fixing isomorphism $\sigma\xi\to \xi$. This isomorphism must necessarily permute both the vertices $v_i$ having $L_i=\emptyset$ and their adjacent edges.  Form a new leg fixing isomorphism $\sigma \eta\to \eta$ by fixing each of the $v_i$ with $L_i=\emptyset$ and their adjacent edges, but applying $\phi$ on all other adjacent edges and vertices. The result is a leg fixing isomorphism which restricts to $\sigma\eta\to \eta$, establishing $I_\xi=I_\eta$.  Finally, there is an isomorphism $det(\xi)_{Aut(\xi)} \cong det(\eta)_{Aut(\eta)}$ as representations of the group $I_\xi=I_\eta$.  This follows from the fact that each of the removed edges was marked, so effectively has degree $1-1=0$.  In particular the automorphism appearing in $Aut(\xi)$ but not in $Aut(\eta)$ act by the identity.

{\bf Subcase: $b_0 = 0$.}  Define $h$ to be the number of indices $1\leq i \leq k$ for which $|L_i|=1$ and $e(v_i) > 2$.  Define $t = b_1-h$.  We proceed by induction on $h$.

The base case is $h=0$, hence $k = b_2+b_1+b_0 = b_2+t$ in this case.  Notice on the one hand that $\sum_{i=1}^k e(v_i) \geq r$, and on the other hand $\sum_{i=1}^k (e(v_i)-1) \leq\beta$ (where $\beta:=g-1$).  We thus conclude $r \leq k+\beta$, hence  %Let $s$ denote the number of indices $1\leq i \leq k$ for which $e(v_i)\geq 2$ and $|L_i| \geq 2$.  Let $q$ denote the number of indices $1\leq i \leq k$ for which $e(v_i)=1$.   Let $t$ denote the number of indices $1\leq i \leq k$ for which $e(v_i)=2$ and $|L_i| =1$. 
%n general $k \geq s+t+q$, with inequality occuring if there exists a vertex $v_i$ with $|L_i|\leq 1$ and $e(v_i)\geq 3$. We proceed by induction on $k-s-t-q$.  The base case assumes $k = s+t+q$ which in turn implies
\begin{equation}\label{blocks}
	r\leq t+b_2 + \beta.
\end{equation}

To find an upper bound for $\rho_\xi$, consider the restriction of  $A_\xi$ to an $S_{n-t}\times S_t$-module, where $S_t$ permutes the $t$ legs associated to those blocks of size 1 corresponding to exactly two parallel edges.  By assumption $\xi$ has an automorphism which permutes the $t$ pairs of parallel edges which must, by stability considerations, be adjacent to the unique leg in said block.   Since all such edges are marked, this automorphism acts by the identity, $V_t$.

On the other hand, each of the blocks of size at least 2 (of which there are $b_2$) must have two legs adjacent to a common vertex.  In particular, removing the DV and its adjacent legs results in a (non-rooted) forest, and so this follows simply by stability considerations.  So repeating the argument from the base case (above), the maximum number of rows in the $S_{n-t}$ module corresponding to permuting all but the $t$ isolated legs has at most $n-t-b_2$ rows.

Applying the Littlewood-Richardson rule to the induced representation from $S_{n-t}\times S_t$ to $S_n$ we see that $A_\xi$ has at most $n-t-b_2+1$ rows, i.e.\ $\rho_\xi \leq n-t-b_2+1$.  We then proceed to bound $\rho_\xi+n$. Since $\xi$ is a core graph, Lemma $\ref{satlemma}$ ensures $ n \leq 3\beta + 2(n-r)$, hence:
$$\rho_\xi+n \leq (n-b_2-t+1) + (3\beta+2(n-r)) 
= 3\beta +3(n-r) +r-b_2-t +1 \leq 4\beta +3(n-r)+1$$
where the last inequality holds due to Equation $\ref{blocks}$.  To conclude the base case, we observe that $4\beta+1 \leq 4\beta + \lceil .5\beta \rceil = \lceil 4.5\beta \rceil$ since $\beta>0$.

Having proven the base case, in which $h=0$, we proceed with the induction step.  If $h>0$ then we may choose a neutral vertex $v_j$ with  $e(v_j) \geq  3$ and $|L_j|=1$.   Define $e:=e(v_j)$. Removing the vertex $v_j$, its $e$ adjacent edges and its one adjacent leg forms a new core graph, call it $\zeta$, to which the induction hypothesis applies. Since each of the removed edges is marked, the core graph $\zeta$ is of type $(g-e+1,n-1,r-e)$.  In addition, since $\zeta$ was created by removing one leg from $\xi$, we must have $\rho_{\xi} \leq \rho_{\zeta}+1$.  Thus the induction hypothesis implies
$$\rho_{\xi}+n\leq \rho_{\zeta} + (n-1)+2 \leq \lceil 4.5(g-e) \rceil  + 3(n-1-r+e)+2 $$
%hence
%$$\rho_{\xi} +n   \leq  4.5(g-e)   + 3(n-r)-.5+3e = 4.5(g-1) +3(n-r)-1.5e+4$$
Using the fact that $e\geq 3$, we can easily deduce
$\rho_{\xi} +n \leq 4.5(g-1) +3(n-r)+.5$
which implies the claimed upper bound.%as desired.
\end{proof}

\begin{corollary}\label{boundcor} Fix $g$ and $\ell$ with $m:= 3(g-1) + 2\ell \geq 0$.  The representation stable sequence of chain complexes $B(g,n,n-\ell)^\top$ stabilizes at or before $\lceil \frac{3m}{2} \rceil$.
\end{corollary}
\begin{proof}  For each degree $0\leq i \leq m$, the consistent sequence $B(g,n,n-\ell)_i$ can be written as a direct sum of consistent sequences $\widehat{A}_\xi$ by Corollary $\ref{dscor}$.  These sequences are indexed by core graphs $(g,k,u)$ with $k\leq n$ and $k-u\leq \ell$.  Thus, by Proposition $\ref{sharpprop}$, the conjugate of each $\widehat{A}_\xi$ stabilizes at or before $\lceil \frac{9(g-1)}{2} \rceil +3(k-u)\leq \lceil \frac{3m}{2} \rceil$.  Applying Lemma $\ref{directsumlem}$ to this direct sum decomposition establishes the result.	 \end{proof}

\subsection{Sharpness of the Bound}\label{fpss}

It remains to prove that the bound established in Corollary $\ref{boundcor}$ is sharp.  Continue with $g$ and $\ell$ satisfying $m:= 3(g-1) + 2 \ell \geq 0$.    By Lemma $\ref{directsumlem}$ it is sufficient to show that there exists a core graph, call it $\theta$, of type $(g,m,m-\ell)$ and satisfying $\rho_\theta + m = \lceil \frac{3m}{2} \rceil$.  We proceed in two cases depending on the (necessarily opposite) parities of $g$ and $m$.

{\bf Case: $g$ is odd and $m$ is even}. Consider the core graph $\theta_{g,\ell}:=\theta_{g,\ell}(0)$ of Definition $\ref{theta}$.  This is a core graph of type $(g,m,m-\ell)$, and it remains to calculate $\rho_{\theta_{g,\ell}}$.  The $S_m$ representation $A_{\theta_{g,\ell}}$ is the induced representation of a 1-dimensional representation of the hyper-octahedral group $S_2\wr S_{m/2}$.  This one dimensional representation is characterized by transpositions acting by +1 and permutation of pairs also acting by +1, so it the trivial one dimensional representation.  The irreducibles appearing in this induced representation are all the doubles, i.e.\ partitions of the form $2\lambda$ where for a partition $\lambda$ of $m/2$, and each double appears with multiplicity $1$ \cite[Proposition 2.3'(ii)]{KT}.  The irreducible decomposition with the most rows appearing in this induced representation is the double of the alternating (so $V_{2^{m/2}}$), which has $m/2$ rows. We thus conclude that the conjugate of $\widehat{A}_{\theta_{g,\ell}}$ stabilizes at $m+m/2$ as desired.  %Observe that this case applies when $m=0$, by interpreting $\theta_{1,0}$ to be the marked graph with no legs and one vertex, provided that we interpret $S_0$ as the trivial group, whose lone irreducible is indexed by Young diagrams with 0 rows.

{\bf Case: $g$ is even and $m$ is odd}. Form a core graph $\theta_{g,\ell}:= \theta_{g,\ell}(1)$ as in Definition $\ref{theta}$.  In this case, $\theta_{g,\ell}$ has a unique unmarked leg adjacent to a pair of marked parallel edges, so the $S_m$ representation $A_{\theta_{g,\ell}}$ is the induced representation of the $S_{m-1}$ representation  $A_{\theta_{g-1,\ell+1}}$ of the previous case. Thus the maximum number of rows appearing in an irreducible is one more than the maximum number of rows appearing in the former, which we saw in the previous case was $3(\beta-1)/2+\ell+1$.  We thus find the stability bound
$
m+\rho = 3\beta + 2\ell + 3(\beta-1)/2 + \ell+1+1 = 4.5\beta + 3\ell +.5
$
for the conjugate of the core sequence $\widehat{A}_{\theta_{g,\ell}}$ as desired.

\begin{remark}  The above proof shows that among the graphs $\theta_{g,\ell}(p)$ constructed in Definition $\ref{theta}$, the cases $p=0,1$ witness the sharpness of the stability bound.  We remark that for $g\geq 3$ and odd, it is also possible to witness this bound by $\theta_{g,\ell}(2)$.  A computation with the Littlewood-Richardson rule, however, shows that these are the only such core graphs whose associated sequences witness the bound. %, since $p\geq 3$ would imply the existence of at least 3 pairs of parallel edges terminating at legs, whose permutation acts trivially, and the Littlewood Richardson rule tells us
\end{remark}

\section{Homology classes via sharp stability.}\label{stabsec}
We now turn to the implication of our main theorem on homology and the proof of Theorem $\ref{odd2}$.  As Theorem $\ref{3h}$ above indicates, knowing precisely where a representation stable sequence of chain complexes stabilizes can be used to glean some information about the multiplicities in homology.  In our example of interest we can further exploit the fact that the highest degree is the last to stabilize to say more.
  
Continue with $g,\ell$ fixed such that $m:=3(g-1) + 2\ell \geq 0$.  From the work we've already done, the first two statements of Theorem $\ref{odd2}$ are rather straight-forward.  The first is:
\begin{lemma}  Let $\lambda = (\lambda_1\cdc\lambda_k)$ be a partition of $\lceil 3m/2\rceil$.  If $k <\lceil m /2 \rceil$ then the multiplicity of $(\lambda,1^{n-\lceil 3m/2 \rceil })$ in $H_i(B(g,n,n-\ell))$ is $0$ for all $i$.
\end{lemma}
\begin{proof}
By Corollary $\ref{legscor}$ we know that, every marked graph of type $(g,\lceil 3m/2 \rceil,\lceil 3m/2 \rceil-\ell)$ has at least $\lceil m/2 \rceil$ marked legs.  Let $\gamma$ be such a marked graph having $h\geq \lceil m/2 \rceil$  marked legs and let $\xi=\epsilon(\gamma)$ be its underlying core graph.  

Then $\widehat{A}_\xi(\lceil 3m/2 \rceil) = A_\xi \circ V_{1^{h}}$, hence every irreducible supported on $A_\xi(\lceil 3m/2 \rceil)$ has at least $h$, and hence at least $\lceil m / 2 \rceil$, rows.  In particular there are no copies of $\lambda$ appearing in $B(g,\lceil 3m/2 \rceil ,\lceil 3m/2 \rceil -\ell)$.  Since $B(g,n,n-\ell)\tensor V_{1^n}$ is representation stable sharply at $\lceil 3m/2 \rceil$, it follows that the multiplicity of $(\lambda,1^{n-\lceil 3m/2 \rceil })$ is $0$ on $B(g,n,n-\ell)$, and hence on homology as well.
\end{proof}

The second statement of Theorem $\ref{odd2}$ is:
\begin{lemma}  Let $\lambda = (\lambda_1\cdc\lambda_k)$ be a partition of $\lceil 3m/2\rceil$ with $k=\lceil m /2 \rceil$.  Let $i<m$.  Then the multiplicity of $(\lambda,1^{n-\lceil 3m/2 \rceil })$ in $H_i(B(g,n,n-\ell))$ is $0$.
\end{lemma}
\begin{proof}  It is enough to prove that the multiplicity of $\lambda$ in   $B(g,\lceil 3m/2 \rceil,\lceil 3m/2\rceil-\ell)_i$ is $0$.  
Suppose otherwise.  Then $\lambda$ would be supported on some summand $\widehat{A}_\xi(\lceil 3m/2\rceil)$ corresponding to a core graph $\xi$ of type $(g,q,r)$ for some $q \leq m$.  If $q<m$ then $\lambda$ being supported on $\widehat{A}_\xi(\lceil 3m/2 \rceil) = A_\xi\circ V_{1^{\lceil 3m/2 \rceil - q}}$ would imply that $\lambda$ has at least $\lceil 3m/2\rceil -q> \lceil m/2 \rceil$ rows, a contradiction.  So it must be the case that $q=m$, but this in turn implies that $\xi$ is of the form $\theta_{g,\ell}(p)$ (by Lemma $\ref{satlemma}$) and hence of degree $m$, contradicting our assumption on $i$. 
\end{proof}

We now turn to the final statement of Theorem $\ref{odd2}$.  We first recall that our notation and conventions regarding partitions and sums of partitions were fixed in Section $\ref{repstabsec}$.  In particular, we regard partitions as finite, non-increasing sequences of positive integers.  In this section we will also have occasion to add partitions to an arbitrary finite sequence of integers by the same convention.  That is, we always pad by zeros as needed and add the sequences point-wise.  If the result is a non-increasing sequence of non-negative integers which ends in zeros, we refer to it as a partition, tacitly disregarding these zeros without further ado.  For example we may regard $(4,3,1)+(0,0,1,0)$ as the partition $(4,3,2)$.

\begin{definition}\label{lambdadef2} Let $(y,p)$ be a pair of non-negative integers such that $2y+p=m$. %and let $n\geq \lceil 3m/2 \rceil$.  
We define $\Lambda(y,p)$ to be the set of partitions of $\lceil 3m/2 \rceil$, counted with multiplicity, of the form $1^{\lceil m/2 \rceil}+\eta+\pi$ where:
	\begin{itemize}
		%	\item  $1^{n-m}$ is the partition of $n-m$ all of whose entries are $1$.
		
		\item $\eta$ is a partition of $2y$, all of whose entries (block sizes) are even.
		
		\item $\pi$ is a sequence of $y+1$ non-negative integers which sum to $p$.  % these represent the number of boxes added to the given row by the Pieri rule, including the possibility of creating at most one new row.

		\item The pair $\eta$ and $\pi$ satisfy the Pieri rule: $\pi_{i+1}+\eta_{i+1} \leq \eta_i$.
	\end{itemize}
%Note that $\eta$ has at most $y$ blocks, and to add $\eta$ and $\pi$ we pad $\eta$ with zeros.  
The last condition implies that $\eta+\pi$ forms a partition of $m$ (after disregarding possible padded zeros).    	We allow the case $y=0$, hence the empty partition of $0$, which would in turn force $\pi= (p,0,...,0)$.  We also allow $p=0$ in which case adding $\pi$ has no effect and the Pieri condition is vacuous.  Finally we observe that each partition of the form $\eta+\pi$ has at most $\lceil m/2 \rceil$ blocks, hence all partitions in $\Lambda(y,p)$ have exactly $\lceil m/2 \rceil$ blocks.  
\end{definition}

For $n \geq \lceil 3m/2\rceil$ we define the $S_n$ representation
$$X(y,p)(n):= \ds\bigoplus_{\lambda\in \Lambda(y,p)} V_{\lambda,1^{n-\lceil 3m/2 \rceil}}. $$
%where $(\lambda, 1^{n-\lceil3m/2\rceil})$ is the partition of $n$ found by padding $\lambda$ with $1$'s. 
%In words, $X(y,p)(n)$ is the sum of irreducibles found by padding partitions in $\Lambda(y,p)$ with 1's.  
Finally, define the $S_n$-representation 
$$\mathsf{Stab}(g,n,n-\ell) = \ds\bigoplus_{\substack{m=2y+p \\ p< g}} X(y,p)(n).$$

\begin{theorem}\label{stab}  Let $\lambda=(\lambda_1\cdc \lambda_k)$ be a partition of $\lceil 3m/2 \rceil$ such that $k=\lceil m/2 \rceil$.  For each $n\geq \lceil 3m/2 \rceil$ the partition $(\lambda, 1^{n-\lceil 3m/2 \rceil})$ has the same multiplicity in $\mathsf{Stab}(g,n,n-\ell)$ as it has in $H_m(B(g,n,n-\ell))$.
\end{theorem}

\begin{proof} % We remark that by Corollary $\ref{multstab}$ it would be enough to prove the case where $n=\lceil 3m/2 \rceil$, although the proof is the same for $n\geq \lceil 3m/2 \rceil$.
Having fixed $g,\ell,m$ as above we choose a pair of non-negative integers $(y,p)$ with $m=2y+p$ and $p<g$.  This ensures $p$ satisfies the conditions of Definition $\ref{theta}$, and we define the core graph $\theta:=\theta_{g,\ell}(p)$ as in this definition.  In particular, this $\theta$ is a graph of type $(g,m,m-\ell)$.

Recall (Lemma $\ref{indlem}$) the $S_m$-representation $A_\theta$ is an induced representation as follows.  First, for any leg labeling of the graph $\theta$ we can consider the group of permutations of $S_m$ which fix the graph, called $I_\theta$.  In this case $I_{\theta} \cong (S_2 \wr S_y) \times S_p$.  The action of $I_\theta$ on $det(\theta)_{Aut(\theta)}$ is trivial on both factors, since the action only permutes marked edges (which have degree $1-1=0$) or unmarked legs.
  	
  	Therefore, writing 1 for the trivial representation of each factor, the $S_m$ representation $A_{\theta} \cong Ind_{S_2 \wr S_y \times S_p}^{S_m}(1\boxtimes 1) $ has the following description.  We first induce $Ind_{S_2 \wr S_y}^{S_{2y}}(1)$ to find the sum of all partitions $\eta$ of $2y$ consisting of only even entries, \cite[Proposition 2.3' (ii)]{KT}.  We then induce the result $Ind_{S_{2y}\times S_p}^{S_{m}}(-)$, by applying the Pieri rule.  The irreducible decomposition of the resulting $S_m$-module is indexed precisely by those partitions of the form $\eta+\pi$ where the pair $(\eta,\pi)$ is as described in Definition $\ref{lambdadef2}$.

  	Since the irreducible decomposition of $A_\theta$ is indexed by such $(\eta,\pi)$, this implies that for each such $(\eta,\pi)$ the partition $\eta+\pi+1^{n-m}$ indexes a corresponding irreducible of $\widehat{A}_{\theta}(n):= A_\theta \circ V_{1^{n-m}}$.   Hence each partition in $\Lambda(y,p)$ indexes a corresponding irreducible factor of $\widehat{A}_{\theta}(\lceil 3m/2 \rceil)$, from which we find
  	$$
  	X(y,p)(n) \subset \widehat{A}_{\theta}(n) \subset B(g,n,n-\ell).
	$$
	We remark that this first inclusion is far from being an equality, rather $X(y,p)$ corresponds to the irreducibles having the minimum possible number of rows (aka blocks).  Put another way, it corresponds to only those terms in $A_\theta \circ V_{1^{n-m}}$ formed by adding one box to each row in the associated Young diagram.
	
	We now apply the differential of $B(g,n,n-\ell)$ to $X(y,p)(n)$.  Let $\theta\wedge (n-m)$ denote the unmarked graph formed by adjoining $n-m$ marked legs to $\theta$.  Since every flag adjacent to the DV of $\theta$, hence $\theta\wedge (n-m)$, is marked, the only possible differential terms supported on $\widehat{A}_{\theta}(n)$ correspond to edge contraction.  Note however $d(\widehat{A}_{\theta}(n))$ is supported on graphs having a minimum of $n-m+1$ marked legs.  Indeed, every edge of $\theta \wedge (n-m)$ is marked, and upon contraction this marking can only result in either a double marked tadpole, which would be zero by definition, or an additional marked leg.  Therefore every irreducible appearing in $d(\widehat{A}_{\theta}(n))$ has at least $n-m+1$ rows.  However, each irreducible appearing in $X(y,p)(n)$ has at most $n-m$ rows.  Thus $d(X(y,p)(n))=0$, and so $X(y,p)(n)$ is contained within the $d$-cycles in top degree, hence contained within the homology. %And since $X(y,p)$ is also concentrated in the top degree, namely $3g+2\ell$, it supports homology. 

	Summing over all $p$, we thus conclude that the multiplicity of $(\lambda,1^{n-\lceil 3m/2 \rceil})$ in $\mathsf{Stab}(g,n,n-\ell)$ is less than or equal to its multiplicity in $H_m(B(g,n,n-\ell))$.  But if they weren't equal there would have to be a core graph $\xi$ distinct from each $\theta_{g,\ell}(p)$ such that $\widehat{A}_\xi(n)$ supports a copy of $(\lambda,1^{n-\lceil 3m/2 \rceil})$.  This could only happen if $\xi$ was of type $(g,m,m-\ell)$, otherwise every irreducible would have more than $\lceil m/2 \rceil$ rows, which in turn would contradict Lemma $\ref{satlemma}$.
	\end{proof}

Note that the third statement of Theorem $\ref{odd2}$ follows from this result.  In particular the Littlewood-Richardson coefficient $N_{\lambda^\prime, 2\eta, p}$ appearing in that statement counts the ways in which $\lambda$ can be written as a sum $1^{\lceil m/2\rceil}+2\eta+\pi$ according to the conditions of Definition $\ref{lambdadef2}$.

\subsection{Examples.}

Combining Corollary $\ref{3h2}$ and Theorem $\ref{odd2}$, we can compute the multiplicity of any irreducible $\lambda$ in $H_m(B(g,n,r))$ provided that $(g,n,r)$ falls in the stable range and that $\lambda$ has at most $n-m$ rows.  Let us give several examples.  As above we denote this piece of the homology by $\mathsf{Stab}(g,n,r)$.  In what follows we modify the notation, denoting $V_\lambda$ simply by $\lambda$.

{\bf Example: Excess 3.}  When $m=2y+p=3$ we find $(y,p)=(0,3)$ or $(1,1)$.  We then tabulate the relevant partitions of $\lceil 3m/2\rceil = 5$ to be 
$\Lambda(1,1) = \{(4,1), (3,2)\}$, and 
$\Lambda(0,3) = \{(4,1)\}$. 
%The partitions appearing in $\Lambda(1,1)$ are found by adding a partition $\eta$ to a sequence $\pi$ having $y+\lceil p/2 \rceil=2$ terms.  After padding with zeros they are $\eta =(2,0)$ and $\pi = (1,0)$ or $(0,1)$, hence  $1^{2}+\eta+\pi$ can be $(4,1)$ or $(3,2)$.  
%The partitions appearing in $\Lambda(0,3)$ are found by adding sequences with $y+\lceil p/2 \rceil=2$ terms.  After padding with zeros, the only possibilities are $\eta=(0,0)$ and $\pi = (3,0)$, hence $1^{2}+\eta+\pi= (4,1)$.
We thus find:
$$\mathsf{Stab}(g,n,r) = 2(4,1,1^{n-5})\oplus (3,2,1^{n-5})\hookrightarrow H_3(B(g,n,r))$$ for $n \geq \lceil 3m/2\rceil =5$ and $g \geq 3$.  In the case $g=2$ the condition $p<g$ implies that  $(4,1,1^{n-5})\oplus (3,2,1^{n-5})$ will appear in the homology.  Comparing this result to the computations of \cite[Section 4.4]{PW} confirms this result in the case $r=11$.

{\bf Example: Excess 4.} When $m=2y+p=4$ we find $(y,p)=(0,4)$, $(1,2)$, or $(2,0)$.  We then tabulate the relevant partitions of $\lceil 3m/2\rceil = 6$ to be 
 $\Lambda(0,4)= \{(5,1)\}$,
$\Lambda(1,2) = \{ (5,1), (4,2), (3,3) \}$, and
$\Lambda(2,0) =\{(5,1), (3,3) \}$. We thus find that if $(g,n,r)$ satisfies $m=4$, $n\geq 6$ and $g>4$:
$$
\mathsf{Stab}(g,n,r) = 
3(5,1,1^{n-6})\oplus(4,2,1^{n-6}) \oplus 2(3,3,1^{n-6}) \hookrightarrow H_{4}(B(g,n,r)).
$$
Imposing the condition on summands that $p<g$, we can also draw conclusions in genus $1$ and $3$.  For example if $g=3$, $m=4$ and $n\geq  6$ then
$$
\mathsf{Stab}(3,n,r) = 2(5,1^{n-5})\oplus(4,2,1^{n-6}) \oplus 2(3,3,1^{n-6}).
$$
%and if $g=1$, $m=4$ and $n\geq 6$ then
%$$
%\mathsf{Stab}(1,n,r) = (5,1^{n-5}) \oplus (3,3,1^{n-6}).
%$$
Specializing to the case $r=11$ we find stable terms $(g,n,r) = (1,13,11),(3,10,11)$ and $(5,7,11)$, in which case these coefficients are confirmed by \cite[Theorem 1.2]{Burk}.   Specializing to the case $r=15$ we get an additional stable term of excess $4$, namely $(g,n,r)=(7,8,15)$.  Applying this calculation in that case recovers the example discussed in the introduction.
% and hence an injection:
%$$
%3V_{5,1^3}\oplus V_{4,2,1^3}\oplus 2V_{3,3,1^2}\hookrightarrow H_{4}(B(7,8,15) ).
%$$ 
%On the other hand, \cite[Theorem 1.2]{Burk} computes (in our degree conventions) $$2V_{2,2}\cong H_{4}(B(7,4,11)).$$  So, as expected, the homology of $B(7,8,15)$ can not be viewed as a formal consequence of the homology of $B(7,4,11)$, since these two terms straddle the stable range. 

{\bf Example: Excess 7.} As a final example, we consider the implications of Theorem $\ref{stab}$ in an example of excess beyond the explicit computations of \cite{PW}, \cite{Burk}, \cite{CLPW}.

Suppose $(g,n,r)$ is of excess $m=2y+p=7$.  We find $(y,p)=(1,5), (2,3)$ or $(3,1)$.  We then tabulate the relevant partitions of  $\lceil 3m/2 \rceil = 11$ (with multiplicity).  They are:

$\Lambda(1,5) = \{(8,1^3), (7,2,1,1), (6,3,1,1) \}$

$\Lambda(2,3) = \{(8,1^3), (7,2,1,1), 2(6,3,1,1), (5,4,1,1), (5,3,2,1), (4,3,3,1)\}$

%4 -> 8,1^3 + 7,2,1,1 + 6,3,1,1 + 5,4,1,1

%2,2 -> 6,3,1,1 +5,3,2,1 + 4,3,3,1
$\Lambda(3,1) = \{(8,1^3), (7,2,1,1),  (6,3,1,1),  (5,4,1,1),  (5,3,2,1),  (4,3,3,1),  (3,3,3,2)\}$

We thus conclude that if $(g,n,r)$ satisfies $m=7,n\geq 11$ and $g > 5$ then
\begin{align*}
\mathsf{Stab}(g,n,r) =  3(8,1^3,1^{n-11}) \oplus 3(7,2,1^2,1^{n-11}) \oplus  4(6,3,1^2,1^{n-11})  \oplus 2(5,4,1^2,1^{n-11}) \oplus \\  2(5,3,2,1,1^{n-11})   \oplus 2(4,3^2,1,1^{n-11}) \oplus (3^3,2,1^{n-11})
\hookrightarrow H_{7}(B(g,n,r))	
\end{align*}
with similar statements, restricting to the $p< g$ summands, in the cases $g=2$ and $g=4$.  If one were to specialize to the case $r=11$ or $r=15$, these statements would apply when $(g,n,r)= (2,13,11)$, $(2,17,15)$, $(4,14,15)$ and $(6,11,15)$.

\section{$B(1,n,r)$ and Whitehouse modules.}
To conclude, we compute $H_\ast(B(1,n,r))$ as an $S_n$ representation.  We then observe that, in genus $g=1$, our sharp stability result may be interpreted as a lift of the sharp stability result of Hersh and Reiner \cite{HR}.

The case $r=0,1$ being trivial we restrict attention to $r\geq 2$ from here on.  The first step will be to compute the homology of $B(1,n,r)$ as an $S_{n-1}$ representation.  This first step has a precursor in \cite{WardStir}, which calculated the homology of a chain complex which is chain homotopic to $B(1,n,r)$, but we opt to give a self-contained proof in our current context.  In what follows we denote $Res^{S_n}_{S_{n-1}}$ simply by $\downarrow$, and we write $C(\mathbb{R}^3,n)$ for the topological space of $n$ distinct labeled points in $\mathbb{R}^3$.

\begin{theorem}\label{n-1}
	$$ \downarrow H_i(B(1,n,r)) = 
	\begin{cases} H^{2(n-r)}(C(\mathbb{R}^3,n-1))\tensor sgn_{n-1} & \text{ if } i =2(n-r) \\
		0 & \text{ else} \end{cases}
	$$
\end{theorem}
% degrees -- B(1,n,r) has homology in degree 2n-2r+3 --- it's W(n,r-1) = H_{n-r+1-1}

\begin{proof}  
	Fix $n$ and $r$.  Recall that $B(1,n,r)$ is described as a direct sum indexed by isomorphism classes of marked graphs $\mathsf{t}$ of type $(1,n,s)$ over all $s\geq r$.  Such a marked graph is necessarily a tree (since $g=1$ implies $\beta=g-1=0$), therefore has no automorphisms.  For each such marked tree, let us choose to consider the leg labeled by $n$ as the root of the tree, and we call the vertex adjacent to the root the root vertex.  It may or may not be the same as the distinguished vertex.  This choice endows the tree $\mathsf{t}$ with a directed structure allowing us to refer to the input and output flags at each vertex.  In particular there is a unique output flag at each vertex and a unique path connecting the distinguished and root vertices (which may be empty if these two vertices coincide).  Edges on this path are said to be ``below'' the DV.
	
	Starting from such a marked graph $\mathsf{t}$ of type $(1,n,s)$ with $E$ edges and distinguished vertex $\mathsf{dv}$, we define several auxiliary statistics.  First $p(\mathsf{t})$ is the number of edges below $\mathsf{dv}$.  Second $\nu(\mathsf{t})$ is defined to be $0$ unless both $|\mathsf{dv}|:=|a^{-1}(\mathsf{dv})|>r$ and the output of $\mathsf{dv}$ is marked, in which case $\nu(\mathsf{t})$ is defined to be $1$.  From these statistics, we define a subcomplex $K \subset B(1,n,r)$ to be the indexed by marked trees $\mathsf{t}$ having $p(\mathsf{t})\neq 0$ or  $|\mathsf{dv}|> r$.  To check that this is indeed a subcomplex, note that the operation $\partial_e$, which contracts an edge $e$, can take $p$ from being non-zero to $p=0$ only by increasing the valence of $\mathsf{dv}$.

		Let us prove that the subcomplex $K$ has no homology.  For this we define $\eta(\mathsf{t}) = 2|E|-p+\nu - s$.  We then filter $K$ by declaring filtration degree $k$ to consist of summands $det(\mathsf{t})$ indexed by those marked graphs $\mathsf{t}$ satisfying $\eta(\mathsf{t}) \leq k$.

	To verify that this is indeed a filtration we compare the value of $\eta$ on a tree $\mathsf{t}$ and the terms indexing the image of the differential applied to $det(\mathsf{t})$. 	First, consider the effect of marking a flag $\partial_f$; $2|E|-p$ is unchanged, while $\nu-s$ decreases by $1$, unless the marked flag was the output of $\mathsf{dv}$, in which case it is unchanged.

Next consider the effect of contracting an edge $\partial_e$ (and reallocating the marking when $e$ is marked).   If the edge is not below $\mathsf{dv}$, then $2|E|$ decreases by $2$, $\nu$ increases by at most $1$, while $s$ and $p$ are unchanged.  Hence $\eta$ decreases.  If the edge is below $\mathsf{dv}$, but not adjacent to it then $2|E|-p$ decreases by $1$ and $\nu-s$ is unchanged, hence $\eta$ decreases.  If the edge is below and adjacent to $\mathsf{dv}$, then $2|E|-p$ decreases by $1$, and $\nu$ increases by at most 1.  Thus $\eta$ is non-increasing, and is constant only if $\mathsf{dv}$ had arity $r$ before contraction, hence $e$ was marked, and if the new marking of $\mathsf{t}/e$ is on the output edge of the new distinguished vertex.

	We thus conclude that $\eta$ does indeed define a filtration on $B(1,n,r)$, and moreover that the differential preserves filtration degree only when marking on unmarked output edge or when contracting an output edge of a distinguished vertex having minimal valence $r$, and then marking the new output edge at the distinguished vertex of the resulting tree.  This in turn implies that the associated graded of $(K,\eta)$ splits as $\mathsf{gr}_\eta(K) = K_1\oplus K_2$, where $K_1$ is indexed by marked trees having exactly $r$ markings and marked output and $K_2$ is indexed by those trees having more than $r$ markings or unmarked output (or both). 

The summands $K_1$ and $K_2$ in turn split into summands which are manifestly acyclic.  To see this let us write $f_{\text{out}}$ for the output of $\mathsf{dv}$.  The summand $K_1$ splits into summands of the form
$$0\to det(\mathsf{t}) \stackrel{\cong}\to det(\mathsf{t}/e\wedge f_{\text{out}}) \to 0$$
by pairing a tree whose distinguished vertex is of minimal valence (which must then have $p>0$ to belong to $K$) with the result of contracting the edge below it and marking the output.  In particular, in the associated graded, this is the unique differential term emanating from the source, while no differential term emanates from the target.  Hence the homology of this summand is $0$.  Likewise, $K_2$ splits into summands of the form
$$0\to det(\mathsf{t}) \stackrel{\cong}\to det(\partial_{f_{\text{out}}}\mathsf{t}) \to 0$$
Hence the homology of this summand is $0$.

We thus conclude that the associated graded of $(K,\eta)$, hence $K$ itself, has no homology.  Thus to compute the homology of $B(1,n,r)$, it suffices to compute the homology of the quotient by $K$.  

This quotient is spanned by those trees having $p(\mathsf{t})=0$ and $|\mathsf{dv}|=r$, and with differential terms which preserve these statistics.  Each such tree specifies a partition of the set $\{1\cdc n-1\}$ into $r-1$ blocks, with each block of size at least $2$, corresponding to the connected components of the forest formed by erasing $\mathsf{dv}$ and its adjacent flags.  The differential terms which survive in this quotient preserve this partition, so we have a direct sum decomposition of the resulting quotient complex over such partitions.

The homology of a summand corresponding to a partition $\pi$, call it $\op{H}_\pi$, is concentrated in maximally expanded degree.  This is seen by comparing each branch of the tree with the bar construction of the commutative operad, in particular if $\pi=\{b_i\}$ then $\op{H}_\pi \cong \tensor_i (Lie(b_i)\tensor sgn)$, concentrated in maximally expanded degree.

Define $S_\pi\subset S_{n-1}$ to be the subgroup which fixes the partition $\pi$.  The group $S_\pi$ is generated by elements which permute the entries in a block of $\pi$ along with those partitions which exchange entries in two blocks of the same size.  The action of $S_\pi$ on $\op{H}_\pi$ is determined on a block by the fact $Lie(m) = Ind_{Z_m}^{S_m}(\zeta_m)$.  The action when exchanging two blocks of size $k$ requires exchanging $(k-2)$ pairs of unmarked edges, along with one pair of marked edges, so the result is $(-1)^{k-2}$.  Notice that the permutation which exchanges these two blocks of size $k$ consists of $k$ transpositions so has parity $(-1)^k$.  Thus, acting by a permutation in $S_\pi$ which exchanges blocks of the partition simply multiplies by the sign of the permutation.

We may therefore describe $\op{H}_{\pi}$ as an induced representation as follows.  Define a permutation $\sigma\in S_{n-1}$ to be a product of cycles $\sigma_i$ which cyclically permute the elements in block $b_i$.  The centralizer of $\sigma$, call it $Z(\sigma)$, is a subgroup of $S_\pi$ and $\op{H}_\pi \cong Ind_{Z(\sigma)}^{S_{\pi}}(Y)$, where $Y$ is the 1-dimensional representation of $Z(\sigma)$ which acts by $\zeta \tensor sgn$ on a block and by $sgn$ when exchanging blocks.  Summing over all such partitions, we have a description of $\downarrow H_{2(n-r)}(B(1,n,r))$ as 
$$\bigoplus_{[\sigma]} Ind^{S_{n-1}}_{Z(\sigma)}(Y)$$
with direct sum taken over conjugacy classes $[\sigma]$ in $S_{n-1}$ having $r-1$ cycles.  To conclude, it suffices to compare this description to the well known description of the $S_{n-1}$ module $H^{2(n-r)}(C(\mathbb{R}^3,n-1))$ as an induced representation.  This description, follows from the work of Cohen in \cite{CLM} and is the odd dimensional analog of the description given by Lehrer and Solomon in \cite{LS}, see also \cite{SW} and \cite{HR}.  Namely,  $H_{2(n-r)}(C(\mathbb{R}^3,n-1))$ is isomorphic to $\oplus Ind_{Z(\sigma)}^{S_{\pi}}(X)$, indexed as above but where $X$ acts by $\zeta_k$ on a $k$-cycle and acts trivially when exchanging cycles of equal length.  Tensoring this description with the sign representation completes the proof.
\end{proof}

\subsection{Lift to the Whitehouse modules}

Having computed $H_\ast(B(1,n,r))$ as a graded $S_{n-1}$ module (Theorem $\ref{n-1}$) we now turn to the $S_n$-module structure.

To describe it, we first recall that Gerstenhaber and Schack \cite{GS} defined a decomposition of the regular representation of the symmetric group $S_n$ via Eularian idempotents.  The regular representation of $S_n$ is the restriction of an $S_{n+1}$-action and subsequently, Whitehouse \cite{Whitehouse} defined a compatible lift of this Eularian decomposition as $S_{n+1}$-modules.  

We use the notation $\mathsf{W}_{n,k} := F_{n}^{(k)}$, for these Whitehouse modules, with the latter being defined in \cite[Definition 1.3]{Whitehouse}.  In particular $\mathsf{W}_{n+1,k}$ is a representation of the symmetric group $S_{n+1}$ of dimension $s_{n,k}$, the Stirling number of the first kind, i.e.\ the number of the permutations in $S_{n}$ having $k$ cycles.

The Whitehouse modules have a topological interpretation, see \cite[Proposition 2]{ER}, which in particular identifies $\downarrow \mathsf{W}_{n+1,k}\cong H^{2(n-k)}(C(\mathbb{R}^3,n))\tensor V_{1^n}$.  Combining this with Theorem $\ref{n-1}$ above tells us $$\downarrow \mathsf{W}_{n,r-1} \cong H^{2(n-r)}(C(\mathbb{R}^3,n-1))\tensor V_{1^{n-1}} \cong \downarrow H_{2(n-r)}(B(1,n,r)),$$
and so the Whitehouse modules are a natural candidate for the lift of $\downarrow H_{2(n-r)}(B(1,n,r))$ to the full $S_n$ action. 

%%%%%%%%%%%%%%%%%%%
%Writing $C(\mathbb{R}^3,n)$ for the configuration space of $n$ points in $\mathbb{R}^3$, we recall that the ungraded vector space underlying $H_\ast(C(\mathbb{R}^3,n))$, is the regular representation of $S_n$.  Under this identification, the Eularian decomposition coincides with the decomposition of $H_\ast(C(\mathbb{R}^3,n))$ by homological degree.

%The space $C(\mathbb{R}^3,n)$ has an $S_{n+1}$ action by considering $\mathbb{R}^3\subset S^3=SU_2$.   Adding a point at $\infty$ gives us a map $C(\mathbb{R}^3,n)\to C(SU_2,n+1)$.  Act by $S_{n+1}$ via permutation of the coordinates and then project to the quotient by the diagonal action.  Choosing the representative of this quotient with final point at $\infty$ gives us a homeomorphism $C(SU_2,n+1)/SU_2 \to C(\mathbb{R}^3,n)$ which defines the action of $S_{n+1}$.  The Whitehouse modules are the conjugate of $H_d(C(\mathbb{R}^3,n))$ along with this $S_{n+1}$-action.  Specifically $\mathsf{W}_{n+1,k}\cong H_{2(n-k)}(C(SU_2,n+1)/SU_2)\tensor V_{1^{n+1}}$.
%%%%%%%%%%%%%%%%%%%%

% B(1,n,r) is quasi-isomorphic to S_{n-1, r-1} which is W_{n,r-1}   the betti number of B(1,n,r)

\begin{theorem}  The homology of the chain complex of $S_n$-modules $B(1,n,r)$ is
	$$  H_i(B(1,n,r)) =
	\begin{cases} \mathsf{W}_{n,r-1} & \text{ if } i =2(n-r) \\
		0 & \text{ else} \end{cases} $$
\end{theorem}

\begin{proof}  We first observe that it suffices to prove that there exists an isomorphism of $S_n$-modules:
\begin{equation}\label{recur}
\downarrow H_i(B(1,n+1,r)) \cong H_i(B(1,n,r-1)\oplus (H_{i-2}(B(1,n,r))\tensor V_{n-1,1})
\end{equation}	
for all $n,r$. % and with $i = 2(n-r)+2$.  
The sufficiency follows from \cite[Proposition 1.4]{Whitehouse} as observed in \cite[Proposition 1]{ER}.  To establish this isomorphism, we compare the characters.  Let $\sigma\in S_n$, and consider $\chi_\sigma$ of both sides.  First, if $\sigma$ has a fixed point, then we can assume without loss of generality that $\sigma$ fixes $n$.  In light of Theorem $\ref{n-1}$, Equation $\ref{recur}$ can be rewritten as a known statement about the characters of the configuration space.  See eg \cite[Proposition 1]{ER}.
	
So it remains to compare the characters on each side of Equation $\ref{recur}$ when $\sigma\in S_n$ has no fixed points.  Since the homology of each chain complexes in Equation $\ref{recur}$ is concentrated in a single even degree, it is enough to show that chain complexes computing the homology on the respective sides have the same virtual character.

We therefore endeavor to compare the character of $\sigma$ of
$\downarrow B(1,n+1,r) \text{ and } B(1,n,r-1)\oplus (B(1,n,r)\tensor V_{n-1,1})$.  Since $\sigma$ has no fixed points, $\chi_\sigma(V_n\oplus V_{n-1,1}) = 0$, and hence it suffices to show that the virtual character of $\downarrow B(1,n+1,r)$ and $B(1,n,r-1)- B(1,n,r)$ coincide at $\sigma$.

By the proof of Theorem $\ref{n-1}$, the chain complex $\downarrow B(1,n+1,r)$ is quasi-isomorphic to a quotient having only those trees with DV adjacent to the root (by convention the leg labeled by $n+1$) and valence $r$ (so every flag at the DV is a marked flag).  The difference on the right hand side, on the other hand, can be described as those trees having exactly $r-1$ markings, and the comparison is made just by erasing the leg labeled by $n+1$ on the left hand side.
\end{proof}

To conclude, we observe that since the homology of $B(1,n,r)$ is concentrated in the top degree, the injective stabilization maps $\psi$ pass to injective maps on homology. % and that the sharp stable bound passes to the homology as well.
 Theorem $\ref{mainthm}$ thus implies that for each $\ell\geq 0$, the sequence of Whitehouse modules $\mathsf{W}_{n,n-\ell-1}\tensor sgn_{n}$ forms a consistent sequence which is representation stable, stabilizing sharply at $n=3\ell$.  The unique term responsible for this sharp stabilization is a copy of $V_{3^\ell}\subset \mathsf{W}_{3\ell,2\ell-1}$.  Restricting the representation stable sequence $\mathsf{W}_{n,n-\ell-1}\tensor sgn_{n}$ from $S_{n}$ to $S_{n-1}$ in each position we find the representation stable sequence $H^{2\ell}(C(\mathbb{R}^3,n-1))$.  This restriction will stabilize at the introduction of $V_{3^\ell,1}\subset \mathsf{W}_{3\ell,2\ell}$.  Reindexing, we conclude that  $H^{2\ell}(C(\mathbb{R}^3,n))$ stabilizes at $n = 3\ell$, recovering the odd dimensional case of the result of Hersh and Reiner \cite[Theorem 1.1]{HR}.

%\bibliography{modularbib}
%\bibliographystyle{alpha}

\end{document}